\documentclass[11pt]{amsart}
\usepackage{amsmath}
\usepackage{mathtools}
\usepackage{amssymb}
\usepackage{amsthm}
\usepackage{mathrsfs}
\usepackage{mathdots}
\usepackage[pagebackref,colorlinks=true,linkcolor=blue,citecolor=blue]{hyperref}
\usepackage{tikz-cd}

\UseRawInputEncoding

\tikzset{
  symbol/.style={
    draw=none,
    every to/.append style={
      edge node={node [sloped, allow upside down, auto=false]{$#1$}}}
  }
}

\newcommand{\N}{\mathbb{N}}
\newcommand{\Z}{\mathbb{Z}}

\newcommand{\R}{\mathbb{R}}
\newcommand{\BC}{\mathbb{C}}

\renewcommand{\H}{\mathbb{H}}

\newcommand{\OO}{{\mathcal O}}

\newcommand{\SL}{\mathrm{SL}}
\newcommand{\GL}{\mathrm{GL}}
\newcommand{\SO}{\mathrm{SO}}
\newcommand{\rO}{\mathrm{O}}
\newcommand{\Sp}{\mathrm{Sp}}
\newcommand{\RU}{\mathrm{U}}
\newcommand{\St}{\mathrm{St}}
\newcommand{\Witt}{\mathrm{Witt}}

\newcommand{\Inv}{\mathrm{inv}}

\newcommand{\Cent}{\mathrm{Cent}}

\newcommand{\disc}{\mathrm{disc}}
\newcommand{\Tran}{\mathrm{Tran}}
\newcommand{\Speh}{\mathrm{Speh}}

\newcommand{\RG}{\mathrm{G}}

\newcommand{\Ind}{\mathrm{Ind}}
\newcommand{\Jac}{\mathrm{Jac}}

\newcommand{\half}[1]{\frac{#1}{2}}

\newcommand{\comment}[1]{}

\newtheorem{thm}{Theorem}[section]
\newtheorem{cor}[thm]{Corollary}
\newtheorem{lemma}[thm]{Lemma}
\newtheorem{prop}[thm]{Proposition}
\newtheorem {conj}[thm]{Conjecture}

\newtheorem {ques/conj}[thm]{Question/Conjecture}

\newtheorem{defn}[thm]{Definition}
\newtheorem{remark}[thm]{Remark}

\newtheorem{assumption}[thm]{Working Hypotheses}

\newtheorem*{globalcond*}{Global Condition}
\newtheorem*{localcond*}{Local Condition}

\newtheorem*{globalconj*}{Global Conjecture}
\newtheorem*{localconj*}{Local Conjecture}

\newtheorem*{nonzero*}{Conjecture on the non-vanishing of the normalized intertwining operators}
\newtheorem*{holo*}{Conjecture on the holomorphicity of the normalized intertwining operators}

\DeclareMathOperator{\Ad}{Ad}

\DeclareMathOperator{\Aut}{Aut}

\DeclareMathOperator{\Gal}{Gal}

\numberwithin{equation}{section}

\let\oldbullet\bullet
\renewcommand{\bullet}{{\vcenter{\hbox{\tiny$\oldbullet$}}}}

\begin{document}

\title[Anti-tempered packets and a lemma of Arthur]
{On anti-tempered local Arthur packets and a lemma of Arthur}

\author{Baiying Liu}
\address{Department of Mathematics\\
Purdue University\\
West Lafayette, IN, 47907, USA}
\email{liu2053@purdue.edu}

\author{Chi-Heng Lo}
\address{Department of Mathematics\\
Purdue University\\
West Lafayette, IN, 47907, USA}
\email{lo93@purdue.edu}

\author{Freydoon Shahidi}
\address{Department of Mathematics\\
Purdue University\\
West Lafayette, IN, 47907, USA}
\email{freydoon.shahidi.1@purdue.edu}

\date{\today}


\thanks{The research of the first named author is partially supported by the NSF Grant DMS-1848058. The research of the third named author is partially supported by the NSF Grant DMS-2135021.}


\begin{abstract}
In this paper, following Arthur's ideas, we rework the process of constructing the anti-tempered local Arthur packets for quasi-split classical groups and their pure inner forms. In particular, we present explicit examples illustrating certain gap in a consequential lemma of Arthur and provide a uniform modification,
based on the work of M{\oe}glin, Waldspurger, and Xu. 
\end{abstract}

\maketitle

\section{Introduction}
An important theme in the theory of automorphic forms is to study the discrete spectrum. A seminal work of Arthur (\cite{Art13}) classifies the discrete spectrum of quasi-split symplectic and orthogonal groups $\mathrm{G}$ into a disjoint union of global Arthur packets, parameterized by global Arthur parameters. These global Arthur packets are patched up by the local Arthur packets, consisting of certain unitary representations of $\mathrm{G}$ over local fields, parameterized by local Arthur parameters. Arthur's work has been extended by Mok (\cite{Mok15}) to quasi-split unitary groups, by Kaletha-Minguez-Shin-White (\cite{KMSW14}) to pure inner forms of unitary groups, { and by M{\oe}glin-Renard (\cite{MR18}) to pure inner forms of special orthogonal and unitary groups}. The local Arthur packets are defined via local character relations, using trace formula method. In the process, an important step is to construct the anti-tempered local Arthur packets from the tempered ones for non-Archimedean local fields by applying the Aubert-Zelevinsky duality operator. It was noticed recently that in this step, there is a gap in a consequential lemma (\cite[Lemma 7.1.1]{Art13}, \cite[Lemma 8.2.2]{Mok15}). 
The purpose of this paper is to rework the process of constructing the anti-tempered local Arthur packets for quasi-split classical groups and their pure inner forms, following Arthur's ideas. In particular, we present explicit examples illustrating the gap in Arthur's lemma and provide a uniform modification,
based on the work of M{\oe}glin (\cite{Moe06b}), M{\oe}glin and Waldspurger (\cite{MW06}), and Xu (\cite{Xu17a, Xu17b}). The main results will be explicated with details as follows. 

Let $F$ be a non-Archimedean local field. Let $\mathrm{G}_n=\Sp_{2n}, \SO_{2n+1}, \SO_{2n}^{\alpha}$, $\mathrm{U}_n$, be quasi-split classical groups or their pure inner forms, where $\alpha$ is a square class in $F$, and let $G=G_n=\mathrm{G}_n(F)$. For any irreducible admissible representation $\pi$ of $G$, the Aubert-Zelevinsky duality operator is defined as
\begin{align}\label{eq AZ}
    D_G( \pi)&=\sum_{P \supseteq P_0} (-1)^{\dim(A_{P_0}/ A_{P})} \Ind_{P}^G ( \Jac_{P} (\pi))
\end{align}
where $P_0$ is a fixed minimal parabolic of $G$, $P=MN$ ranges over standard parabolic subgroups of $G$, and $A_{P}$ is the maximal $F$-split torus contained in the center of $M$. We shall also rewrite $A_{P}$ by $A_{M}$.
Then, we know that $D_{G}(\pi)= \beta(\pi) \widehat{\pi}$, where $\widehat{\pi}$ is an irreducible representation and 
        \[ \beta(\pi)= (-1)^{\dim(A_{M_0}/ A_{M_{\pi}})},\]
where $M_0$ is the fixed minimal Levi of $G$, and $M_{\pi}$ is the standard Levi subgroup of $G$ such that $\pi$ is a subquotient of $\Ind_{M_{\pi}}^{G}(r_{\pi})$ for some supercuspidal representation $r_{\pi}$ of $M_{\pi}$.   
The representation $\widehat{\pi}$ is called the Aubert-Zelevinsky involution of $\pi$.

Let $\phi$ be a tempered $L$-parameter, which is also regarded as a local Arthur parameter. { Namely, we regard $\phi$ as a homomorphism
\[ \phi: W_F\times \SL_2^D(\BC)\times \SL_{2}^A(\BC) \to {}^L \RG\]
that is trivial on $\SL_2^A(\BC)$.
Let ${S}_{\phi}:= \textrm{Cent}(\mathrm{Im}(\phi),\widehat{\RG}(\BC)))$ and consider its component groups
\begin{align*}
     \mathcal{S}_{\phi}&:= S_{\phi}/ S_{\phi}^{\circ},\\
    \overline{\mathcal{S}}_{\phi}&:= S_{\phi}/ S_{\phi}^{\circ} Z(\widehat{\RG}(\BC))^{\Gamma},
\end{align*}
where $\Gamma$ is the absolute Galois group.

Consider the anti-tempered local Arthur parameter 
\[ \psi: W_F\times \SL_2^D(\BC)\times \SL_{2}^A(\BC) \to {}^L \RG\]
obtained from $\phi$ by swapping the two $\SL_2(\BC)$, i.e.,
\[ \psi(w,x,y):= \phi(w,y,x).\]
Then we can define $S_{\psi},$ $ \mathcal{S}_{\psi}$ and $\overline{\mathcal{S}}_{\psi}$ similarly. Note that $S_{\phi}=S_{\psi}$. Hence, we may identify $  \mathcal{S}_{\phi}\cong \mathcal{S}_{\psi}$ and $ \overline{\mathcal{S}}_{\phi} \cong \overline{\mathcal{S}}_{\psi}$.}

For quasi-split classical groups, under the above identifications and the fact that the map
\[ \Pi_{\phi} \to \widehat{\overline{\mathcal{S}}}_{\phi}\]
is a bijection, for any $\pi \in \Pi_{\phi}$, let $\sigma_{\pi}$ be the distribution corresponding to the same character in $\widehat{\overline{\mathcal{S}}}_{\psi} \cong \widehat{\overline{\mathcal{S}}}_{\phi}$ characterized by the system of equations given by the endoscopic transfer \cite[(7.1.2)]{Art13}.
Then, Arthur and Mok proved the following lemma on the relation between $\sigma_{\pi}$ and $\widehat{\pi}$. 

\begin{lemma}[{\cite[Lemma 7.1.1]{Art13}, \cite[Lemma 8.2.2]{Mok15}}]\label{lem Arthur}
   Let $G_n$ be a quasi-split classical group.  For any $\pi \in \Pi_{\phi}$, we have
    \[ \langle s_{\psi}, \pi \rangle \sigma_{\pi} = \beta (\phi) \beta(\pi) \widehat{\pi}  \]
    in the Grothendieck group, where $s_{\psi}$ is the image of $\psi(1,1,-1)$ in the component group $\mathcal{S}_{\psi}$,
    \[ \beta(\phi):= (-1)^{\dim(A_{M_0}/ A_{M_{\phi}})},\] 
$M_0$ is the minimal Levi of $G_n$ and $M_{\phi}$ is the minimal Levi of $G$ for which the $L$-group ${}^L M_{\phi}^0$ contains the image of $\phi$.
\end{lemma}

By \cite[Theorem 2.2.1]{Art13} and \cite[Theorem 3.2.1]{Mok15}, Lemma \ref{lem Arthur} is equivalent to the following corollary.

\begin{cor}\label{cor Arthur}
    Let $\phi$ be a tempered $L$-parameter of $G_n$ and let $\psi:= \widehat{\phi}$.
    Then for any $\pi \in \Pi_{\phi}$, 
    \begin{enumerate}
        \item $\sigma_{\pi}=\widehat{\pi}$.
        \item $\langle s_{\psi}, \pi \rangle = \beta (\phi) \beta(\pi)$.
    \end{enumerate}
\end{cor}

However, recently, based on computations of explicit examples (see \S \ref{sec example}), we realized that in the setting of Lemma \ref{lem Arthur} and Corollary \ref{cor Arthur}, $\pi \in \Pi_{\phi}$ and $\widehat{\pi}$ may not always correspond to the same character in $\widehat{\overline{\mathcal{S}}}_{\phi}$, which implies that
Lemma \ref{lem Arthur} and Corollary \ref{cor Arthur} are not correct in general. 

In this paper, following suggestions of Xu, based on the work of M{\oe}glin (\cite{Moe06b}), M{\oe}glin and Waldspurger (\cite{MW06}), and Xu (\cite{Xu17a, Xu17b}), we provide corrections for Lemma \ref{lem Arthur} and Corollary \ref{cor Arthur}. For simplicity, we only state the version for quasi-split classical groups in the introduction and refer to Theorem \ref{thm main} for the precise statements in the cases of pure inner forms.

\begin{thm}[Theorem \ref{thm main}]\label{thm main intro}
Let $\phi$ be a tempered local Arthur parameter of a quasi-split classical group $G_n$ and let $\psi= \widehat{\phi}$. For any $\varepsilon \in \widehat{\overline{\mathcal{S}}}_{\phi}$,  we have
    \begin{align}\label{eq final mod intro}
        \varepsilon( s_{\psi}) \pi(\psi, \varepsilon \varepsilon^{M/MW}_{\psi}) = \beta (\phi_{\psi}) \beta(\pi(\phi,\varepsilon)) \widehat{\pi(\phi, \varepsilon)} 
    \end{align}
    in the Grothendieck group, where
    the character $\varepsilon^{M/MW}_{\psi}$ is defined in Lemma \ref{lem MW/W}. 
\end{thm}

From Theorem \ref{thm main intro}, if $\pi \in \Pi_{\phi}$ corresponds to $\varepsilon \in \widehat{\mathcal{S}}_{\phi}$, then $\widehat{\pi}$ should correspond to $\varepsilon \varepsilon^{M/MW}_{\psi}$.  The character $\varepsilon^{M/MW}_{\psi}$ is not always trivial. In the case of $\varepsilon^{M/MW}_{\psi}$ being trivial, i.e., $\pi$ and $\widehat{\pi}$ correspond to the same character, we need to use $\beta(\phi_{\psi})$ (instead of $\beta(\phi)$) on the right hand side of the equation (see the Example in \S \ref{example 2}).

Theorem \ref{thm main intro} is proved 
by reworking the process of constructing the anti-tempered local Arthur packets of $G_n$ (see Conjecture \ref{conj 2.2.1}, Remark \ref{rmk obtain packets}, and Theorem \ref{thm main}), which is reduced to the following three steps:

\begin{enumerate}
\item (Proposition \ref{prop twisted}) Let $\phi$ be a tempered local Arthur parameter of $G_n$ and $\psi=\widehat{\phi}$. We verify the equality
\begin{align}\label{eq twisted AZ intro}
    D_{G_n}(\eta_{\phi})= \beta(\phi_{\psi}) \eta_{\psi},
\end{align}
where $\eta_{\phi}$ (resp. $\eta_{\psi}$) is the stable distributions associated to $\phi$ (resp. $\psi$) characterized by twisted endoscopic transfer identity. When $G_n$ is symplectic or special orthogonal groups, based on \cite[Appendix]{Xu17b}, the equality \eqref{eq twisted AZ intro} is reduced to the computation of Aubert-Zelevinsky involution of certain representation $\pi^+$ of a disconnected group $\GL_N^+(E)$ associated to the classical group $G_n$, which is done in \cite[\S 3]{MW06} (also see \cite[\S 6.3]{Xu17b}). The argument also works for unitary groups.  For completeness, we provide a uniform proof for quasi-split classical groups in \S \ref{sec working hypothesis}, following \cite{MW06, Xu17b}.

    \item (Lemma \ref{lem MW/W}) Let $\psi$ be an anti-tempered local Arthur parameter of $G_n$. 
   Given an endoscopic data $(G',s, \xi)$ (see \S \ref{sec endoscopic}), 
   $\psi$ factors through and produce a local Arthur parameter $\psi'$ of $G'$. 
   Then the product of signs 
    \[ e(G_n)\alpha(G_n,G') \beta(\phi_{\psi}) \beta(\phi_{\psi'})\]
    only depends on the image of $s$ in $\mathcal{S}_{\psi}$. Here $e(G_n)$ is the Kottwitz sign of $G_n$. Moreover, the product of sign is a character of $\overline{\mathcal{S}}_{\psi}$ trivial on $s_{\psi}$, denoted by $\varepsilon^{M/MW}_{\psi}$. We verify this step by direct computation.
   \item (Proposition \ref{prop eq of sign Moeglin}) For any $\varepsilon \in \widehat{\mathcal{S}}_{\phi}$, we have that
  \begin{align*}
      \beta(\pi(\phi,\varepsilon)) \beta(\phi_{\psi}) \varepsilon(s_{\psi})=1.
  \end{align*} 
  {
  We first verify this equality for supercuspidal representations in $\Pi_{\phi}$ using M\oe glin's parametrization of supercuspidal representations (\cite[Theorem 2.5.1]{Moe11} and \cite[Theorem 3.4]{MR18}). Then, the general case of tempered representations is proved by induction based on the Jacquet module of tempered representations.}
\end{enumerate}

{
Suppose $G$ is a non-classical group where local Langlands correspondence is known, e.g., $G=G_2$ by \cite{AX22, GS23}. Then, the ingredients in the two steps above are well-defined. Suppose further that these statements (Lemma \ref{lem MW/W}, Proposition \ref{prop eq of sign Moeglin}) hold for $G$. Then our argument provides a construction/definition of anti-tempered local Arthur packets of $G$. See \S \ref{sec non-classical gp} for more details.}

At the end of this paper, we compute the $L$-parameter of the Aubert-Zelevinsky involution of generic representations of quasi-split classical groups following the idea of \cite{Jan18} under an assumption (Working Hypothesis \ref{assu AZ}). For symplectic and split odd special orthogonal groups, this assumption has been verified in \cite{Ato20}. Other cases follow similarly and will be verified explicitly in future work. This result has its own interest and is closely related to the enhanced Shahidi conjecture (see \cite{HLLZ22,LS23}) and the upper bound conjecture of wavefront sets of representations (see \cite{HLLS24}). See Remark \ref{remark generic dual} for more discussion. 

Following is the structure of this paper. In \S \ref{sec classical groups}, we introduce the classical groups considered in this paper. In \S \ref {sec parameters}, we recall the preliminaries on local Arthur parameters, local $L$-parameters, components groups and their characters, and endoscopic groups. In \S \ref{sec example}, we provide two detailed examples  on groups $\Sp_2(F)$ and $\SO_3(F)$, respectively, illustrating the gaps in Lemma \ref{lem Arthur} and Corollary \ref{cor Arthur}. In \S \ref{construction of anti-tempered}, following Arthur's ideas, we construct anti-tempered local Arthur packets for all quasi-split classical groups and their pure inner forms uniformly. Along the way, we fix the gap in \cite[Lemma 7.1.1]{Art13} and provide the modification. Then, we discuss the generalization of the construction to non-classical groups. In \S \ref{sec MW/W} and \S \ref{sec eq of sign}, we prove Lemma \ref{lem MW/W} and Proposition \ref{prop eq of sign Moeglin}, respectively, which are two important steps in our construction of anti-tempered local Arthur packets. In \S \ref{sec working hypothesis}, we provide details for Proposition \ref{prop twisted}. In \S \ref{dual of generic reps}, we compute the $L$-parameter of the Aubert-Zelevinsky involution of generic representations of quasi-split classical groups. 

\subsection*{Acknowledgements} 
The authors would like to thank Dihua Jiang and Bin Xu for their interests, helpful comments and suggestions, and James Arthur for encouragements and helpful communications. The authors also would like to thank Hiraku Atobe for helpful communications. The second named author is grateful to the hospitality of the National Center for Theoretical Sciences during his visit and
helpful discussions with Cheng-Chiang Tsai, in which the examples in \S \ref{sec example} were motivated.

\section{Classical groups}\label{sec classical groups}
In this section, we specify the classical groups considered in this paper.

\subsection{Quadratic space}

Let $E=F$ or a quadratic extension $F(\delta)$ of $F$, and let $\sigma \in \Gal(E/F)$ be the trivial element in the first case and the non-trivial element in the second case. Let $(V,q_V)$ be a finite dimensional vector space $V$ over $E$ equipped with a $\varepsilon$-Hermitian form form $q_V$, where $\varepsilon \in \{\pm 1\}$. That is, for any $v,w,u \in V$ and $\alpha, \beta \in E$, we have 
\begin{align*}
    \begin{cases}
        q_V(\alpha v +\beta w, u)= \alpha q_V(v,u)+ \beta q_V(w,u),\\
    q_V(u,v)= \varepsilon(v,u)^{\sigma}.
    \end{cases}
\end{align*}
We shall sometimes abbreviate $(V,q_V)$ by $V$ if there is no confusion. If $E=F$, we say $V$ is orthogonal if $\varepsilon=1$ and symplectic if $\varepsilon=-1$. If $E\neq F$, we say $V$ is Hermitian if $\varepsilon=1$ and skew-Hermitian if $\varepsilon=-1$. 

We shall consider the following invariants of $V$. Let $\H$ denote the hyperbolic plane, i.e., $\H = Ev + E v^{\ast}$ with bilinear form $q_{\H}(v,v)=q_{\H}( v^{\ast},v^{\ast})=0$ and $q_{\H}(v, v^{\ast} )=1.$ The Witt index of the quadratic space $(V,q_V)$ is an integer $\Witt(V)=\mathfrak{r}$ such that $(V,q_{V}) \cong \H^{\mathfrak{r}} \oplus (V_{an}, q_{an})$ where $(V_{an}, q_{an})$ is anisotropic. 

Let $\mathfrak{n}:= \dim_E(V)$. Take an orthogonal basis $\{e_{1},\dots, e_{\mathfrak{n}}\}$ of $V$ with
\[ q_V( e_i, e_i)= d_{i}.\]
Then the discriminant of $(V,q_V)$ is given by
\[ \disc(V):= (-1)^{\half{\mathfrak{n}(\mathfrak{n}-1)}} \prod_{i=1}^{\mathfrak{n}} d_i \in \begin{cases}
\delta^{\mathfrak{n}} \cdot  F^{\times}/ \N E^{\times } & \text{ if }E \neq F, \varepsilon= -1,\\
    F^{\times}/ \N E^{\times } & \text{ otherwise,}
\end{cases} \]
where $\N E^{\times }= \{x \sigma(x) \ | \ x \in E^{\times }\}$. Let $\epsilon(V)$ be the Hasse invariant of $V$. Thus, if $E = F$, let $(\cdot, \cdot)_F$ denote the Hilbert symbol, then 
\[ \epsilon(V):=\prod_{i <j} (d_i,d_j)_F \in \{ \pm 1\}.\]
If $E \neq F$, then 
\[ \epsilon(V):= \begin{cases}
   ( \disc(V), \delta^2 ) & \text{ if }\varepsilon \neq -1,\\
   ( \delta^{-\mathfrak{n}}\disc(V), \delta^2) & \text{ if }\varepsilon =-1.
\end{cases}\]

If $E=F$, then the isometric class of the non-degenerate quadratic space $(V,q_V)$ is uniquely determined by $\dim(V)$, $\disc(V)$ and $\epsilon(V)$. If $E \neq F$, then the isometric class of the non-degenerate quadratic space $(V,q_V)$ is determined by $\dim(V)$ and $\disc(V)$ (and $\varepsilon$).

\subsection{Classical groups}
Let $(V,q_V)$ be a quadratic space considered in the previous subsection. 
The classical groups considered in this paper are $G=G(V):=\textrm{Isom}(V,q_V)^{\circ}$, the identity component of the group
\[\textrm{Isom}(V,q_V):=\{ T \in \Aut(V)\ | \ q_V(Tv, Tw)= q_V(v,w), \ \forall v,w \in V \}. \]
Recall $\mathfrak{n}:= \dim_E(V)$. Let $ n:= \left\lfloor \half{\mathfrak{n}}\right\rfloor$, and let $K$ be the splitting field of the quasi-split inner form of $G$. We identify ${}^L G=\widehat{G}(\BC) \rtimes \Gal(K/F)$ as in \cite[\S 7]{GGP12}. 
\begin{center}
    \begin{tabular}{|c|c|c|c|c|}
     \hline 
      \rule{0pt}{3ex} $(E,\varepsilon)$&  $G$& $\widehat{G}$ & $K$ & ${}^L G$\\
     \hline 
    \rule{0pt}{3ex}    $E=F,$& $\SO(V),$ & $\Sp_{2n}(\BC)$& $F$ & $\Sp_{2n}(\BC)$ \\
     $\varepsilon=+1 $& $\mathfrak{n}=2n+1$ & &  &  \\
     \hline
 \rule{0pt}{3ex}     $E=F,$& $\SO(V),$ & $\SO_{2n}(\BC)$& $F(\sqrt{\disc(V)})$ & $\SO_{2n}(\BC)$ if $\disc(V)\in (F^{\times})^2$ \\
     $\varepsilon=+1 $& $\mathfrak{n}=2n$ & &  &  $\rO_{2n}(\BC)$ if $\disc(V)\not\in (F^{\times})^2$ \\
     \hline
     \rule{0pt}{3ex}     $E=F,$& $\Sp(V),$ & $\SO_{2n+1}(\BC)$& $F$ & $\SO_{2n+1}(\BC)$  \\
     $\varepsilon=-1 $& $\mathfrak{n}(V)=2n$ & &  &   \\
     \hline
      \rule{0pt}{3ex}     $E\neq F,$& $\RU(V),$ & $\GL_{\mathfrak{n}}(\BC)$& $E$ & $\GL_{\mathfrak{n}}(\BC) \rtimes \Gal(E/F)$  \\
     $\varepsilon= \pm 1 $& $\mathfrak{n}\in\{2n,2n+1\}$ & &  &   \\
     \hline
\end{tabular}
\end{center}

The map from the isometric class of $V$ to the isomorphism class of $G(V)$ is not an injection. However, fixing the group $G(V)$, we may classify the pure inner forms $G(V')$ of $G(V)$ by the isometric classes of quadratic spaces $V'$ with certain conditions, which we describe case by case below. Note that the $F$-rank of $G(V)$ is equal to $\Witt(V)$.

\begin{enumerate}
    \item [1.] Special odd orthogonal groups: The inner forms of $\SO(V)$ are $\SO(V')$ with $\dim(V')=\dim(V)$ and $\disc(V')=\disc(V)$. If $2n+1= \dim(V)\geq 3$, then there exists a $V'$ satisfying above conditions but $\epsilon(V') \neq \epsilon(V)$, unique up to isometric.

    If $\epsilon(V)=1$, then $\Witt(V)=n$ and $\SO(V)$ is split. In this case if $n\geq 1$, then $\Witt(V')=n-1$ and $\SO(V')$ is not quasi-split.
    
    \item [2.] Symplectic groups: The only pure inner form of $\Sp(V)$ is itself. We have $\Witt(V)=n$ and $\Sp(V)$ is split. 
    
    \item [3.] Special even orthogonal groups: The inner forms of $\SO(V)$ are $\SO(V')$ with $\dim(V')=\dim(V)$ and $\disc(V')=\disc(V)$. If $2n= \dim(V)\geq 4$, then there exists a $V'$ satisfying above conditions but $\epsilon(V') \neq \epsilon(V)$, unique up to isometric.

    If $\disc(V) \in (F^{\times})^2$ and $\epsilon(V)=1$, then $\SO(V)$ is split. In this case if $n\geq 2$, then $ \Witt(V')=n-2$ and $\SO(V')$ is not quasi-split. If $\disc(V) \not\in (F^{\times})^2$, then $\SO(V)$ is quasi-split but not split, and $\SO(V) \cong \SO(V')$ over $F$.

    \item [4.] Unitary groups: The inner forms of $\RU(V)$ are $\RU(V')$ where $V'$ is Hermitian or skew-Hermitian with $\dim(V')=\dim(V)$. There is a Hermitian $V'$ satisfying above condition but $\disc(V) \neq \disc(V')$, unique up to isometric. 
    
    If $\dim_E(V)=2n+1$ is odd, then  $\Witt(V)= n$. The group $\RU(V)$ is quasi-split and $\RU(V) \cong \RU(V')$ over $F$. If $\dim_E(V)=2n$ is even and $\epsilon(V)=1$, then $\Witt(V)=n$ and $\RU(V)$ is quasi-split. In this case, $\Witt(V')=n-1$ and $\RU(V')$ is not quasi-split. 
\end{enumerate}

In this paper, the classical group $G$ is always associated to a quadratic space $(V,q_V)$, and hence we may distinguish $G$ among its pure inner forms.

\subsection{Parabolic subgroups}
Recall that we have 
\[ (V,q_V) \cong \H^{\mathfrak{r}} \oplus (V_{an}, q_{an}),\]
where $\mathfrak{r}= \Witt(V)$ and $(V_{an}, q_{an})$ is anisotropic. Let 
\[ V_{an,r}:= \H^{r} \oplus (V_{an}, q_{an}). \]
Any Levi subgroup of $G(V)$ is isomorphic to 
\begin{align}\label{eq parabolic subgroups}
    \GL_{n_1}(E) \times \cdots \times \GL_{n_{f}}(E) \times G(V_{an,\mathfrak{r}-r }  ),
\end{align}
where $0 \leq r \leq \mathfrak{r}$ and $[n_1,\ldots, n_f]$ is a partition of $r$. A minimal parabolic subgroup corresponds to the partition $[\underbrace{1, \ldots ,1}_{\mathfrak{r}\text{-copies}}]$ of $\mathfrak{r}$.

Suppose $P$ is a parabolic subgroup of $G(V)$ with Levi subgroup isomorphic to \eqref{eq parabolic subgroups}. Let $\tau_i$ be a representation of $\GL_{n_i}(E)$ for $i=1,\ldots , f$ and $\pi$ a representation of $G(V_{an, \mathfrak{r}-r})$. We denote the normalized parabolic induction $\Ind_{P}^G (\tau_1 \otimes \cdots \otimes \tau_f \otimes \pi)$ by
\[ \tau_1 \times \cdots \times \tau_f \rtimes \pi.\]

\section{Parameters, component groups, and endoscopic groups}\label{sec parameters}

In this section, we recall the definition of local Arthur parameters of classical groups and combinatorial descriptions of the components groups and their characters; the definitions of endoscopic groups and local $L$-parameters.

\subsection{Local Arthur parameters}
Let $\RG$ be a connected reductive group defined over $F$ and let $G=\RG(F)$. A local Arthur parameter of $G$ is a continuous homomorphism
\[ \psi: W_F \times \SL_2^D(\BC) \times \SL_2^A(\BC) \to {}^L G\]
with the following conditions.
\begin{enumerate}
    \item [(i)] For any $w \in W_F$, $\psi(w,1,1)$ is semisimple. If $\lambda$ is an eigenvalue of $\psi(w,1,1)$, then $\min(|w|^{-1/2},|w|^{1/2})< |\lambda| < \max(|w|^{-1/2},|w|^{1/2})$.
    \item [(ii)] The projection onto $\Gal(K/F)$ is the natural map $W_F/W_K \to \Gal(K/F)$.
    \item [(iii)] The restriction to both $\SL_2(\BC)$ are algebraic. 
    \item [(iv)]  The homomorphism $\psi$ is $G$-relevant. In other words, if $\textrm{Im}(\psi)$ is contained in some Levi subgroup ${}^L P$ of ${}^L G$, then there is a corresponding parabolic subgroup $\mathrm{P}$ of $\RG$ defined over $F$. 
\end{enumerate}
We shall call $\SL_2^D(\BC)$ the Deligne-$\SL_2(\BC)$ and $\SL_2^A(\BC)$ the Arthur-$\SL_2(\BC)$.
Two local Arthur parameters are equivalent if they are conjugate by an element of $\widehat{\RG}(\BC)$. By abuse of notation, we shall not distinguish between $\psi$ and its equivalence class. 

The local Arthur parameter $\psi$ is called \emph{generic} if $\psi|_{\SL_2^A(\BC)}$ is trivial, and is called \emph{tempered} if further $\psi|_{W_F}$ has bounded image. From each local Arthur parameter $\psi$, we may associate another local Arthur parameter $\widehat{\psi}$ by
\[ \widehat{\psi}(w,x,y):= \psi(w,y,x).\]
The local Arthur parameter $\psi$ is called anti-generic (resp. anti-tempered) if $\widehat{\psi}$ is generic (resp. tempered).

For each local Arthur parameter $\psi$ of $G$, let $S_{\psi}$ denote the centralizer of the image of $\psi$ in $ \widehat{G}(\BC)$. We denote the associated component groups by
\begin{align*}
    \mathcal{S}_{\psi}&:= S_{\psi}/ S_{\psi}^{\circ},\\
    \overline{\mathcal{S}}_{\psi}&:= S_{\psi}/ S_{\psi}^{\circ} Z(\widehat{G}(\BC))^{\Gamma}.
\end{align*}
It is clear that $S_{\psi}=S_{\widehat{\psi}}$,  $\mathcal{S}_{\widehat{\psi}}\cong \mathcal{S}_{\psi}$ and $\overline{\mathcal{S}}_{\widehat{\psi}} \cong \overline{\mathcal{S}}_{\psi}$. For classical groups considered in this paper, these component groups are always abelian 2-groups. 

\subsection{A combinatorial description of the component groups}\label{sec component group}
In this subsection, we give an explicit combinatorial description of $\mathcal{S}_{\psi},$ $\overline{\mathcal{S}}_{\psi}$ and their Pontryagin duals $\widehat{\mathcal{S}_{\psi}}$, $\widehat{\overline{\mathcal{S}}}_{\psi}$ by decomposing the local Arthur parameter $\psi$. First, we recall the construction of the associated representation $\psi_{\GL}$. Then, we recall the computation of its component group and relate it with $\mathcal{S}_{\psi},$ $\overline{\mathcal{S}}_{\psi}$ from \cite[\S 4, 8]{GGP12}

If $G=G(V)$ is a symplectic group or a special orthogonal group, then we fix a standard embedding $\xi: {}^L G \hookrightarrow \GL(M)$ where $M\cong \BC^{2n+1}$ or $M\cong \BC^{2n}$. For each local Arthur parameter $\psi$, we define 
\[\psi_{\GL}:=\xi \circ \psi: W_E \times \SL_2^D(\BC) \times \SL_2^A(\BC) \to \GL(M).\]
If $G= \RU(V)$, for each local Arthur parameter $\psi$, we define 
\[\psi_{\GL}:= \psi|_{W_E \times \SL_2^D(\BC) \times \SL_2^A(\BC)}: W_E \times \SL_2^D(\BC) \times \SL_2^A(\BC) \to \GL(M).\]

The representation $\psi_{\GL}$ is (conjugate-)self-dual in the following sense. Take any $s \in W_{F}$ that generates the quotient $W_F/W_E\cong \Gal(E/F)$. Define a representation $\psi_{\GL}^s$ via conjugating by $s$:
\[ {}^s \psi_{\GL}(w,x,y):= \psi_{\GL}( s w s^{-1},x,y  ).\]
Then, there exists a non-degenerate bilinear form $B$ on $M$ such that for any $m_1, m_2 \in M$ and $ \tau \in W_E \times \SL_2^D(\BC) \times \SL_2^A(\BC)$,
\begin{align}\label{eq conjugate self-dual}
\begin{cases}
    B( \psi_{\GL}(\tau) m_1, {}^s \psi_{\GL}(\tau) m_2 )= B(m_1,m_2),\\
    B(m_1, m_2)= \widehat{\varepsilon} B(m_2, \psi_{\GL}(s^2,1,1) m_1).
\end{cases}
\end{align}
where the sign $\widehat{\varepsilon}=1$ if $G= \Sp(V)$, $\SO(V)$ with $\dim(V)=2n$, or $G= \RU(V)$ with $\dim_E(V)$ being odd, and $\widehat{\varepsilon}=-1$ otherwise (see \cite[Theorem 8.1]{GGP12}). This gives isomorphisms of representations
\begin{align*}
    f&:{}^s \psi_{\GL} \to \psi_{\GL}^{\vee},\\
{}^s f^{\vee}&: {}^s \psi_{\GL} \to {}^s ({}^s \psi_{\GL}^{\vee})  \xrightarrow[]{\psi_{\GL}(s^2)} \psi_{\GL}^{\vee},
\end{align*}
such that $f= \widehat{\varepsilon}\  {}^s f^{\vee}$. The equivalence class of $\psi^{s}$ is independent of the choice of $s$, and hence we may write ${}^\sigma \psi:=\psi^s$. Then, the isomorphisms above show that $\psi \cong {}^\sigma \psi^{\vee}$.

The map $\psi \mapsto \psi_{\GL}$ is a surjection onto the set of local Arthur parameters of $\GL_N(E)$ that is (conjugate-)self-dual with sign $\widehat{\varepsilon}$. If $G$ is not an even special orthogonal group, then the map $\psi \mapsto \psi_{\GL}$ is an injection. If $G$ is an even special orthogonal group, then $(\psi_{1})_{\GL}$ is equivalent to $(\psi_2)_{\GL}$ if and only if $\psi_1$ is equivalent to one of $\{\psi_2, \psi_2^{c}\}$, where $\psi_{2}^c$ is the outer conjugation of $\psi_2$ (\cite[Theorem 8.1(ii)]{GGP12}). Moreover, if $G$ is an even special orthogonal group, then the quadratic character  
\[ \det (\psi_{\GL}): W_{F} \to \{\pm 1\}\]
corresponds to $\disc(V)$, which determines the group $G(V)$ up to isomorphism.

Now we decompose $\psi_{\GL}$ as a direct sum of irreducible representations
\[ \psi_{\GL}= \bigoplus_{i \in I'} (\rho|\cdot|^{x_i} \otimes S_{a_i} \otimes S_{b_i})^{\oplus m_i},\]
where 
\begin{enumerate}
    \item [$\oldbullet$] The representation $\rho_i$ is a self-dual irreducible representation of $W_{E}$ and $x_i \in \R$;
     \item [$\oldbullet$] The representation $S_{y}$ is the $y$-dimensional irreducible algebraic representation of $\SL_2(\BC)$;
     \item [$\oldbullet$] The irreducible representations $\{\rho|\cdot|^{x_i} \otimes S_{a_i} \otimes S_{b_i} \}_{i \in I'} $ are pairwise non-isomorphic, and $m_i$ indicates the multiplicity.
\end{enumerate} 
With this decomposition, the Condition (i) for the local Arthur parameter implies $ |x|<1/2$.

For each $i \in I'$, one can define $ {}^\sigma \rho_i^{\vee}$ and ${}^\sigma(\rho_i|\cdot|^{x_i} \otimes S_{a_i} \otimes S_{b_i})^{\vee}$ similarly. Let $\widehat{\varepsilon}_i=0$ if $\rho_i \otimes S_{a_i} \otimes S_{b_i} \not\cong {}^\sigma(\rho_i|\cdot|^{x_i} \otimes S_{a_i} \otimes S_{b_i})^{\vee},$ which is equivalent to $x_i \neq 0$ or $\rho_i \not \cong  {}^\sigma \rho_i^{\vee}$. Otherwise, let $\widehat{\varepsilon}_i$ denote the sign in \eqref{eq conjugate self-dual}. Consider the decomposition of index set $I'=I_{gp} \sqcup I_{bp} \sqcup I_{nsd}'$ as follows.
\begin{align*}
    I_{nsd}':= \{ i \in I \ | \  \widehat{\varepsilon}_i=0 \},\ 
    I_{bp} :=\{ i \in I\ | \ \widehat{\varepsilon}_i= - \widehat{\varepsilon} \},\ 
     I_{gp} :=\{ i \in I\ | \ \widehat{\varepsilon}_i=  \widehat{\varepsilon} \}.
\end{align*}
Since $\psi$ is (conjugate-)self-dual with sign $\widehat{\varepsilon}$, we may rewrite the decomposition as 
\begin{align}\label{eq decomp psi_GL}
    \psi_{\GL}&= \bigoplus_{i \in I_{gp}} (\rho_i \otimes S_{a_i} \otimes S_{b_i})^{\oplus m_i} + \bigoplus_{i \in I_{bp}} (\rho_i \otimes S_{a_i} \otimes S_{b_i})^{\oplus m_i} \\
    &+\bigoplus_{i \in I_{nsd}} \left(\rho_i |\cdot|^{x_i}\otimes S_{a_i} \otimes S_{b_i}+  {}^{\sigma}\rho_i^{\vee} |\cdot|^{-x_i}\otimes S_{a_i} \otimes S_{b_i} \right)^{\oplus m_i},  
\end{align}
where $I_{nsd}$ is a subset of $I_{nsd}'$ with half of the size. Note that if $i \in I_{bp}$, $m_i$ must be even. Let $\Aut(M,B)$ be the subset of $\GL(M)$ preserving the bilinear form $B$, and let $C_{\psi}$ denote the subset of $\Aut(M,B)$ that centralizes the image of $\psi_{\GL}$. Then 
\[ C_{\psi} \cong \prod_{i \in I_{gp}} \rO_{m_i}(\BC) \times  \prod_{i \in I_{bp}} \Sp_{m_i}(\BC)  \times \prod_{i \in I_{nsd}} \GL_{m_i}(\BC). \]
In particular, 
\[ \mathcal{C}_{\psi}:=C_{\psi}/(C_{\psi})^{\circ} \cong (\Z/2\Z)^{|I_{gp}|} \]
is an abelian 2-group. We shall identify it as the set of functions $e: I_{gp} \to \{ \pm 1\}$, and write 
\[ e(\rho_i \otimes S_{a_i} \otimes S_{b_i}):= e(i).  \]
Let $e_i$ denote the image of a simple reflection in $\rO_{m_i}(\BC)$ (and identity elsewhere) in $\mathcal{C}_{\psi_{\GL}}$, which corresponds to the function
\[ e_i( \rho_j \otimes S_{a_j} \otimes S_{b_j})= \begin{cases}
    -1 & \text{ if }i=j,\\
    1 & \text{ if }i \neq j.
\end{cases}\]
The image of $-\textrm{id}_M \in \Aut(M,B)$ in $\mathcal{C}_{\psi}$, denoted by $e_0$, is given by
\[ e_0(\rho_i \otimes S_{a_i}\otimes S_{b_i}):= (-1)^{m_i},  \]
and the image of $s_\psi:= \psi_{\GL}(1, 1, -1 )$, denoted by $e_\psi$, is given by
\[ e_{\psi}(\rho_i \otimes S_{a_i}\otimes S_{b_i}):= (-1)^{ (b_i-1)m_i }.\]

We identify the Pontryagin dual $\widehat{\mathcal{C}}_{\psi}$ with the set of functions $\varepsilon: I_{gp} \to \{ \pm 1\}$ via the inner product
\[ \langle \varepsilon , e \rangle := \prod_{i \in I_{gp}} \varepsilon(\rho_i \otimes S_{a_i} \otimes S_{b_i}) \ast e(\rho_i \otimes S_{a_i} \otimes S_{b_i}), \]
where 
\[ \varepsilon(\rho_i \otimes S_{a_i} \otimes S_{b_i}) \ast e(\rho_i \otimes S_{a_i} \otimes S_{b_i})= \begin{cases}
    -1 & \text{ if }e(\rho_i \otimes S_{a_i} \otimes S_{b_i})=\varepsilon(\rho_i \otimes S_{a_i} \otimes S_{b_i})=-1,\\
    1 & \text{ otherwise.}
\end{cases}\]
Sometimes we write $\varepsilon(e):= \langle \varepsilon, e \rangle.$
The determinant map $\Aut(M,B) \to  \{\pm 1\}$ induces $\det: \mathcal{C}_{\psi_{\GL}} \to \{ \pm 1 \}:$ 
\[ \det (e) = \prod_{i \in I_{gp}} e(\rho_i \otimes S_{a_i} \otimes S_{b_i})^{\dim(\rho_i \otimes S_{a_i} \otimes S_{b_i})},  \]
which corresponds to the element $\varepsilon_0 \in \widehat{\mathcal{C}}_{\psi}$: 
\[\varepsilon_0( \rho_i \otimes S_{a_i} \otimes S_{b_i}):= (-1)^{\dim(\rho_i \otimes S_{a_i} \otimes S_{b_i}) }.\]
 Let $\mathcal{C}_{\psi}^{+}:= \{e \in \mathcal{C}_{\psi} \ | \ \det(e)=1\}$. Then $ \widehat{\mathcal{C}}_{\psi}^{+}= \widehat{\mathcal{C}}_{\psi}/\varepsilon_0$.
Now we recall the relation between $ \mathcal{C}_{\psi}$ and $\mathcal{S}_{\psi}$.

\begin{thm}[{\cite[Theorem 8.1 (iii)]{GGP12}}]
Let $\psi$ be a local Arthur parameter of $G$. If $G= \RU(V)$, then $\mathcal{S}_{\psi} \cong  \mathcal{C}_{\psi}$. If  $G= \SO(V)$ or $\Sp(V)$, then $\mathcal{S}_{\psi} \cong  \mathcal{C}_{\psi}^{+}$
\end{thm}

Since $ Z(\widehat{\RG}(\BC))^{\Gamma} \subseteq \{ \pm 1\}$ holds for all classical groups we consider, we may identify
\begin{align*}
\widehat{\overline{\mathcal{S}}}_{\psi}&= \{ \varepsilon \in \widehat{C}_{\psi}\ | \ \langle \varepsilon , e_0 \rangle =1  \}\\
&=\{ \varepsilon \in \widehat{C}_{\psi}\ | \ \prod_{i \in I_{gp}} \varepsilon(\rho_i \otimes S_{a_i} \otimes S_{b_i})^{m_i} =1  \}.
\end{align*}
Finally, we evaluate $\varepsilon(s_{\psi})$ explicitly for each $\varepsilon \in \widehat{\mathcal{S}}_{\psi}$, which will be used in \S \ref{sec proof of lemma MW/W}.

\begin{align}
\begin{split}
\varepsilon(s_{\psi})&= \langle \varepsilon, e_{\psi}\rangle\\
  &=\prod_{i \in I_{gp}} \varepsilon(\rho_i\otimes S_{a_i} \otimes S_{b_i}) \ast e_{\psi}(\rho_i\otimes S_{a_i} \otimes S_{b_i})\\
   &=\prod_{i \in I_{gp}} \varepsilon(\rho_i\otimes S_{a_i} \otimes S_{b_i})^{m_i (b_i-1)}. \label{eq eva s_psi}
    \end{split}
\end{align}

\subsection{Endoscopic groups}\label{sec endoscopic}
In this subsection, we recall the definition of endoscopic group determined by a semisimple element $s \in S_{\psi}$ and recall some computation we need.

Let $\psi$ be a local Arthur parameter of $G$ and $s$ be a semisimple element in $ S_{\psi}$. There is a quasi-split reductive group $G'$ such that
\[ \widehat{G'} \cong \Cent(s ,\widehat{G})^{\circ},\]
and the isomorphism extends to 
\[ \xi: {}^L G'\to {}^L G\]
such that $\xi( {}^L G') \subseteq \textrm{Cent}(s, {}^L G)$ and $\psi$ factor through $\xi({}^L G')$. This gives a local Arthur parameter of $G'$, which we denote by $\psi'$. We say the pair $(G', \psi')$ corresponds to $(\psi, s)$ through $\xi$, and denote $(G',\psi') \to (\psi,s)$. The group $G'$ obtained in this way is called an \emph{endoscopic group} of $G$ and the triple $(G',s, \xi)$ is called an endoscopic data. We say $G'$ is elliptic (or $(G',s, \xi)$ is an elliptic endoscopic data) if $Z(\widehat{G'})^{\Gamma}$ is finite. If $G$ is a quasi-split classical group, then any elliptic endoscopic group is a product of at most two quasi-split classical groups. We shall also see this in the computation in \S \ref{sec computation of endoscopic group}.

\subsection{\texorpdfstring{$L$-parameters}{}}

An $L$-parameter $\phi$ of $G$ is a continuous homomorphism 
\[ \phi: W_F \times \SL_2(\BC)  \to {}^L G\]
with the following conditions.
\begin{enumerate}
    \item [(i)] For any $w \in W_F$, $\psi(w,1,1)$ is semisimple. 
    \item [(ii)] The projection onto $\Gal(K/F)$ is the natural map $W_F/W_K \to \Gal(K/F)$.
    \item [(iii)] The restriction to $\SL_2(\BC)$ are algebraic. 
    \item [(iv)]  $\phi$ is $G$-relevant. In other words, if $\textrm{Im}(\phi)$ is contained in some Levi subgroup ${}^L P$ of ${}^L G$, then there is a corresponding parabolic subgroup $P$ of $G$ defined over $F$. 
\end{enumerate}

For each $L$-parameter $\phi$ of $G$, we may define 
\[\phi_{\GL}: W_E \times \SL_2(\BC) \to {}^L G\]
similarly as in \S \ref{sec component group}, which is (conjugate-)self-dual.

For a local Arthur parameter $\psi$ of $G$, we may associate a homomorphism $\phi_{\psi}$ by
\[ \phi_{\psi}(w,x,y):= \psi \left(w,x, \begin{pmatrix}
    |w|^{1/2} & \\ & |w|^{-1/2}
\end{pmatrix} \right).  \]
If it is $G$-relevant, then it gives an $L$-parameter of $G$. Note that $(\phi_{\psi})_{\GL}= \phi_{\psi_{\GL}}$. We remark that it is possible that $\psi$ is $G$-relevant but $\phi_\psi$ is not $G$-relevant. For example, let $G^{\ast}$ be the split $\SO_7(F)$ and $G$ be its non-quasi-split inner form. Let $\rho$ be the trivial representation of $W_F$ and consider the local Arthur parameter $\psi$ of $G$ with 
\[ \psi_{\GL}= (\rho \otimes S_{1} \otimes S_{2})^{\oplus 3}.\]
It is $G$-relevant. However,
\[ \phi_{\psi_{\GL}}= (\rho|\cdot|^{1/2} \otimes S_1)^{\oplus 3}+(\rho|\cdot|^{-1/2} \otimes S_1)^{\oplus 3} \]
is not $G$-relevant since it factors through the Borel subgroup of ${}^L G$.

In the rest of this paper, we do not distinguish between $\phi$ and $\phi_{\GL}$, $\psi$ and $\psi_{\GL}$ by abuse of notation. This simplifies the notation when constructing parameters. For example, let $\phi_0$ be an $L$-parameter of $G(V_{an,r})$ and $\phi_1$ be an $L$-parameter of $\GL_{d}(E)$. We shall write
\[ \phi:= (\phi_1+ {}^\sigma \phi_1^{\vee} )+\phi_0,\]
which means that $\phi$ is an $L$-parameter of $G(V_{an,r+d})$ such that
\[ \phi_{\GL}= (\phi_1+ {}^\sigma \phi_1^{\vee} )+(\phi_0)_{\GL}.\]
Note that in this way we modulo the outer conjugation for special even orthogonal groups. However, this will not affect the argument whenever we use this convention.

\section{Two examples}\label{sec example}

In this section, we provide two examples on groups $\Sp_2(F)$ and $\SO_3(F)$, respectively, 
which serves as counter-examples for Lemma \ref{lem Arthur}. 

\subsection{\texorpdfstring{An example on $\Sp_2(F)$}{}}\label{example 1}
Let $\chi$ be the non-trivial unramified quadratic character of $F$. Consider the following tempered local Arthur parameter of $G=\Sp_2(F)$
\[ \phi:= \chi \otimes S_1 \otimes S_1 + \chi \otimes S_1 \otimes S_1 + 1 \otimes S_1 \otimes S_1. \]
Let $\psi:=\widehat{\phi}$. 
Then $\psi=\phi$ and the component group
$\mathcal{S}_{\psi}=\mathcal{S}_{\phi}$
has order $2$. Fixing a Whittaker datum, we write the local Arthur packet corresponding to $\psi$ as 
\[ \Pi_{\phi}= \Pi_{\psi}= \{ \pi^{+}, \pi^{-} \},\]
where $\pi^{+}$ is generic and unramified (with respect to $\Sp_2(\OO_F)$).

Let $P_0=M_0N_0$ be the Siegel parabolic subgroup of $\Sp_{2}(F)$. Regard $\chi$ as a unitary self-dual supercuspidal representation (a character) of $M_0 \cong \GL_1(F)$. Then as a representation of finite length, we have (See \cite[Proposition 4.2]{Ato22} or the proof of \cite[Proposition 2.4.3]{Art13})
\[ \textrm{Ind}_{P_0}^G (\chi) =  \pi^+  \oplus \pi^-.\]
Since both $\pi^{+}, \pi^{-}$ occurs as a subrepresentation of $\textrm{Ind}_{P_0}^G \chi$, by Frobenius reciprocity, we have
\[ \Jac_{P_0} (\pi^{+}) \geq \chi,\ \  \Jac_{P_0} (\pi^{-}) \geq \chi\ \]
in the Grothendieck group. On the other hand, we have (see \cite[\S 5]{Xu17a} for example)
\[ \Jac_{P_0} ( \Ind_{P_0}^G (\chi)  )= \chi + \chi.\]
 We conclude that 
\[ \Jac_{P_0} (\pi^+)=\chi,\ \ \Jac_{P_0}(\pi^{-})= \chi.\]

Now we compute the Aubert-Zelevinsky involution of $\pi^{+}, \pi^{-}$ from the definition \eqref{eq AZ}. There are only two standard parabolic subgroups, which are $G$ and $ P_0$. We have
\[ \{1\}=A_{G} \subseteq A_{P_0} \cong \GL_1(F), \]
where $A_{P_0}$ is the maximal split torus in the center of the Levi of $P_0$.
\begin{align*}
    D_G( \pi^{+})&=\sum_{P \supseteq P_0} (-1)^{\dim(A_{P_0}/ A_{P})} \Ind_{P}^G ( \Jac_{P} (\pi^{+}))\\
    &= -(\pi^{+})+ \Ind_{P_0}^{G} \chi\\
    &= \pi^{-},
\end{align*}
and similarly
\[ D_G(\pi^{-})= \pi^{+}.\]
Apply the setting of Lemma \ref{lem Arthur} to $\phi, \psi$ and $\pi^+$, we have
\begin{enumerate}
    \item [$\oldbullet$] $\sigma_{\pi^+}= \pi^+$.
    \item [$\oldbullet$] $\beta(\pi^+)=1$.
    \item [$\oldbullet$] $\beta(\phi)= 1$.
    \item [$\oldbullet$] $s_{\psi}$ is trivial.
\end{enumerate}
Then Lemma \ref{lem Arthur} says that 
\[\pi^+= \langle s_{\psi}, \pi \rangle \sigma_{\pi^+}= \beta(\phi) \beta(\pi^+) \widehat{\pi^+}= \pi^-,\]
which is a contradiction.

\subsection{\texorpdfstring{An example on $\SO_3(F)$}{}}\label{example 2}
Consider the tempered local Arthur parameter
\[ \phi= 1 \otimes S_2 \otimes S_1\]
of the split group $G=\SO_3(F)$. Let $\psi = \widehat{\phi}= 1 \otimes S_1 \otimes S_2.$
Then $\mathcal{S}_{\phi}=\mathcal{S}_{\psi}$ is the trivial group.

The tempered local Arthur packet $\Pi_{\phi}$ consists of a single representation $\pi_{gen}$, which is generic but not supercuspidal. The local Arthur packet $\Pi_{\psi}$ consists of a single representation $\widehat{\pi_{gen}}$, the Aubert-Zelevinsky involution of $\pi_{gen}$. We have an exact sequence (for example, see \cite[Proposition 3.4]{Tad20})
\[ 0 \to \pi_{gen} \to \Ind_{\GL_1}^{\SO_3} (|\cdot|^{1/2}) \to \widehat{\pi_{gen}} \to 0. \]
By a similar computation as in \S  \ref{example 1}, we have
\[ D_{G}(\pi_{gen})= \widehat{\pi_{gen}},\]
and hence $\beta(\pi_{gen})=1$.
On the other hand, since the image of $\phi$ is the whole $\Sp_2(\BC)$, $M_{\phi}=G$. Therefore, $ \beta(\phi)= -1.$ Then, Lemma \ref{lem Arthur} says that $\widehat{\pi_{gen}}=-(\widehat{\pi_{gen}})$. which is a contradiction.

\section{\texorpdfstring{Construction of anti-tempered local Arthur packets}{}}\label{construction of anti-tempered}
In this section, we construct anti-tempered local Arthur packets from tempered local Arthur packets for pure inner forms of classical groups in a uniform manner. We follow the strategy in \cite[\S 7.1]{Art13}, and point out two computations of certain product of signs (Lemma \ref{lem MW/W} and Proposition \ref{prop eq of sign Moeglin}), which is crucial to the strategy. We remark that when $G$ is quasi-split symplectic or special orthogonal groups, these two computations are essentially done in \cite{Moe06b, MW06}.

\subsection{Definition of signs}
In this subsection, we associate certain signs to endoscopic groups, $L$
-parameters and irreducible representations.

Let $G=G(V)$ be a classical group considered in \S \ref{sec classical groups} and $G^{\ast}$ be the quasi-split pure inner form of $G$. Let $P_0$ denote a minimal parabolic subgroup of $G$ with Levi subgroup $M_0$. For a parabolic subgroup $P$ with Levi subgroup $M$, let $A_M$ or $A_P$ denote the maximal $F$-split torus contained in the center of $M$. Let $e(G)$ denote the Kottwitz sign
\[ e(G):= (-1)^{r(G)-r(G^{\ast})},\]
where $r(G)$ is the $F$-rank of $G$. These definitions naturally generalize to endoscopic groups of $G$, which are products of classical groups and $\GL_d(E)$. 

Let $M$ be any Levi subgroup of $G=G(V_{an,\mathfrak{r}})$. We compute $A_{M}$ explicitly. Recall that $M$ is isomorphic to 
\[ \GL_{n_1}(E) \times \cdots \times \GL_{n_f}(E) \times G(V_{an, \mathfrak{r}-r}),  \]
where $ 0 \leq r =n_1 +\cdots + n_f \leq \mathfrak{r} = \Witt(V)$. Thus $A_M \cong \GL_1(E)^{\times f}.$ As a consequence, $A_{M_{0}} \cong \GL_1(E)^{\times \mathfrak{r}}$ and $\dim(A_{M_0})= \mathfrak{r} =r(G)$.

 Suppose $G'$ is an endoscopic group of $G$. Define
\[ \alpha(G, G'):= (-1)^{\dim(A_{M_0}/A_{M_0'})},\]
where $M_0'$ is the minimal Levi of $G'$. Note that $\alpha(G, G')= e(G) \alpha(G^{\ast}, G')$.

For an $L$-parameter $\phi$ of $G$, we let $M_{\phi}$ denote a minimal Levi subgroup for which the $L$-group ${}^L M_{\phi}$  contains the image of $\phi$. Then define
\[ \beta(\phi):= (-1)^{\dim( A_{M_0}/ A_{M_{\phi}})}.\]
It does not matter whether we regard $\phi$ as an $L$-parameter of $G$ or $G^{\ast}$. For a local Arthur parameter $\psi$, we define $\beta(\phi_\psi)$ by regarding $\phi_{\psi}$ as an $L$-parameter of $G^{\ast}$.

Finally, for an irreducible representation $\pi$ of $G$, let $\sigma$ be a supercuspidal representation on a Levi subgroup $M_{\pi}$ of $G$ such that $\pi$ is a subquotient of $\Ind_{M_{\pi}}^G \sigma$. Then define 
\[ \beta(\pi):= (-1)^{\dim( A_{M_0}/ A_{M_{\pi}})}.\]

The definition of $\beta(\phi)$ and $\beta(\pi)$ also works for $L$-parameters $\phi$ and irreducible representation $\pi$ of $\GL_n(E)$. We shall denote them by $\beta_{\GL}(\phi) $ and $\beta_{\GL}(\pi)$ to specify the groups in this case.

\subsection{Characterization of local Arthur packets} 
In this subsection, we recall the statement of \cite[Theorem 2.2.1, Conjecture 9.4.2]{Art13} in the setting for pure inner forms, which is formulated in \cite[Theorem* 1.6.1]{KMSW14}.

Recall that our classical group $G=G(V)$ is always associated to a quadratic space $(V,q_V)$. The quadratic space $(V,q_V)$ determines a character $\chi_{V}$ of $Z(\widehat{G})^{\Gamma} \subseteq \{\pm 1\}$. More explicitly, if $-1 \in Z(\widehat{G})^{\Gamma}$, then $\chi_V(-1):= \epsilon(V)$. Let
\[ \widehat{\mathcal{S}}_{\psi,\chi_V}:= \{ \varepsilon \in \widehat{\mathcal{S}}_{\psi} \ | \ \varepsilon(e_0)= \chi_V(-1)  \},\]
which is either $ \widehat{\overline{\mathcal{S}}}_{\psi}$ or $\widehat{\mathcal{S}}_{\psi} \setminus \widehat{\overline{\mathcal{S}}}_{\psi}$.

Let $G'$ be an endoscopic group of $G$. For each stable distribution $S$ on $G'$, we denote the endoscopic transfer of $S$ from $G'$ to $G$ by $\Tran_{G'}^G S$ as introduced in \cite{LS87}. For more details, see \cite[\S 4]{Hir04} for example. The following conjecture characterizes local Arthur packets for $G$. 

\begin{conj}[{\cite[Theorem 2.2.1, Conjecture 9.4.2]{Art13}, \cite[Theorem* 1.6.1]{KMSW14}}]\label{conj 2.2.1}
    Let $\psi$ be a local Arthur parameter of $G=G(V)$.
    \begin{enumerate}
        \item [(a)] For any endoscopic group $G'$ of $G$ and local Arthur parameter $\psi'$ of $G'$, there exists a unique stable distribution $\eta_{\psi'}$ of $G'$ that is compatible with twisted endoscopic transfers (see \S \ref{sec twisted endoscopic transfer}) and products.
        \item [(b-1)]For any $s \in S_{\psi}$ that gives $(G',\psi') \rightarrow (\psi,s)$, the endoscopic transfer $\Tran_{G'}^G(\eta_{\psi'})$ only depends on the image of $s$ in $\mathcal{S}_{\psi}.$ 
        Define $\eta_{\psi,x}:= \Tran_{G'}^G(\eta_{\psi'})$, where $x$ is the image of $s$ in $\mathcal{S}_{\psi}$.
        \item [(b-2)] For each $\varepsilon\in \widehat{\mathcal{S}}_{\psi,\chi_V}$, define distributions $\pi(\psi,\varepsilon)$ via the system of equations
        \[ \eta_{\psi,x}= e(G) \sum_{\varepsilon \in \widehat{\mathcal{S}}_{\psi,\chi_V}} \varepsilon(s_{\psi}x) \pi(\psi, \varepsilon),\]
        where $x$ ranges over $\mathcal{S}_{\psi} $. Then each distribution $\pi(\psi,\varepsilon)$ is a non-negative integral linear combination of character of irreducible representations.
    \end{enumerate}
\end{conj}

In the following remark, we explain how to define the local Arthur packet $\Pi_{\psi}(G(V))$ and the map $\Pi_{\psi}(G(V)) \to \widehat{\mathcal{S}}_{\psi,\chi_V}$ from Conjecture \ref{conj 2.2.1}.
\begin{remark}\label{rmk obtain packets}
    We do not distinguish an irreducible representation and its character in the following discussion. By Conjecture \ref{conj 2.2.1}(b-2), we may write
\[ \pi(\psi, \varepsilon)= \pi_{\varepsilon,1}+\cdots + \pi_{\varepsilon, l_{\varepsilon}}\]
for each $\varepsilon\in \widehat{\mathcal{S}}_{\psi,\chi_V}$, where $\pi_{\varepsilon,i}$ is irreducible, but $\pi_{\varepsilon,i}$, $\pi_{\varepsilon,j}$ are not necessarily distinct and $l_{\varepsilon}$ can be zero. The local Arthur packet $\Pi_{\psi}(G(V))$ is defined to be the multi-set 
\[ \Pi_{\psi}(G(V)) := \bigsqcup_{\varepsilon\in \widehat{\mathcal{S}}_{\psi,\chi_V}} \{ \pi_{\varepsilon,1},\dots, \pi_{\varepsilon, l_{\varepsilon}} \},\]
equipped with a mapping
\begin{align*}
    \Pi_{\psi}(G(V)) &\longrightarrow \widehat{\mathcal{S}}_{\psi,\chi_V},\\
    \pi_{\varepsilon, r} & \longmapsto \varepsilon.
\end{align*}
\end{remark}

The above conjecture is proved when $\psi$ is tempered. See \cite{Art13, Mok15, KMSW14, MR18, Ish23}. For special even orthogonal groups, we modulo the outer conjugation.

\begin{thm}\label{thm tempered}
    Conjecture \ref{conj 2.2.1} holds for any tempered local Arthur parameter $\phi$ of pure inner forms of classical groups $G(V)$. Moreover, for any $\varepsilon \in \widehat{\mathcal{S}}_{\phi,\chi_V}$, the distribution $\pi(\phi,\varepsilon)$ is a character of an irreducible representation.
\end{thm}

\subsection{Construction of anti-tempered packets}
In this subsection, we prove Conjecture \ref{conj 2.2.1} for anti-tempered local Arthur parameters for pure inner form of classical groups. The construction reduces to compute certain product of signs associated to parameters (Lemma \ref{lem MW/W}) and tempered representations (Proposition \ref{prop eq of sign Moeglin}), whose proof will be given in the later sections.

Before the construction, we recall several important results. First, the results of Arthur and Mok show that Theorem \ref{thm tempered} implies that Conjecture \ref{conj 2.2.1}(a) holds for $G$. 

\begin{prop}[{\cite[Lemma 2.2.2]{Art13}, \cite[Proposition 8.2.1]{Mok15} }]\label{prop red 2.2.1(a) to tempered}
Let $G'$ be an endoscopic group of $G$. If Conjecture \ref{conj 2.2.1}(a) holds for any tempered local Arthur parameter of $G'$, then it holds for any general local Arthur parameter of $G'$.    
\end{prop}

 Recall that we let $D_G$ to denote the Aubert-Zelevinsky duality operator on the Grothendieck group of irreducible representations of $G$, which can also be regarded as an operator on the invariant distributions of $G$. Hiraga showed that Aubert-Zelevinsky duality operator is compatible with endoscopic transfer in the following sense.

\begin{thm}[{\cite[Theorem 1.5]{Hir04}}]\label{thm Hiraga}
    Suppose $G'$ is an endoscopic group of $G$ and $S'$ is a stable distribution on $G'$. Then we have
    \[ D_{G}\circ \Tran_{G'}^G (S')= \alpha(G,G') \Tran_{G'}^G \circ D_{G'}(S').\]
\end{thm}

  Xu generalized Hiraga's argument and showed the compatibility of Aubert-Zelevinsky duality operator with twisted endoscopic transfer (\cite[(A.1)]{Xu17b}). Combining with an explicit computation of the disconnected version of $D_G$ on the disconnected $\GL_N(F)$ due to M{\oe}glin and Waldspurger for symplectic and quasi-split special orthogonal groups (see \cite[Corollary 6.13]{Xu17b} and also \cite[\S 3]{MW06}), we have
\begin{align}\label{eq twisted D_G}
    D_{G}(\eta_\phi)= \beta(\phi_{\psi}) \eta_\psi,
\end{align}
 for any tempered local Arthur parameter $\phi$ and $\psi= \widehat{\phi}$. The same argument works for unitary groups.
 For completeness, following \cite{MW06, Xu17b}, we provide 
 in \S \ref{sec working hypothesis} a uniform proof for quasi-split classical groups. We summarize the statement as in the following proposition. 

\begin{prop}\label{prop twisted}
    Let $G'$ be an endoscopic group of a quasi-split classical group $G$ and $\phi'$ be a tempered local Arthur parameter of $G'$. Set $\psi':= \widehat{\phi'}$. The following identity of stable distribution holds:
    \[D_{G'}(\eta_{\phi'})= \beta(\phi_{\psi'}) \eta_{\psi'}.\]
\end{prop}

Now we start the construction. Fix an anti-tempered local Arthur parameter $\psi$  of $G=G(V)$ and let $\phi:= \widehat{\psi}$.  Recall that there are natural identifications between the centralizers and component groups of $\psi$ and $\phi$. Suppose $(G',\psi') \to (\psi,s)$ and $(G',\phi') \to (\phi,s)$ for some $s \in S_{\psi}$. We can compute $\Tran_{G'}^G(\eta_{\psi'})$ from $\Tran_{G'}^G(\eta_{\phi'})$ by Theorem \ref{thm Hiraga} and Proposition \ref{prop twisted} as follows.
\begin{align*}
    \Tran_{G'}^G (\eta_{\psi'})&= \Tran_{G'}^G(\beta(\phi_{\psi'}) D_{G'}(\eta_{\phi'}) )\\
    &=\beta(\phi_{\psi'})  \alpha(G,G')D_{G}\circ \Tran_{G'}^G(\eta_{\phi'}).
\end{align*}

By Theorem \ref{thm tempered}, $D_{G}\circ \Tran_{G'}^G(\eta_{\phi'})$ only depends on the image of $s$ in $\mathcal{S}_{\phi} \cong\mathcal{S}_{\psi}$. Thus, $ \Tran_{G'}^G (\eta_{\psi'})$ only depends on the image of $s$ in $ \mathcal{S}_{\psi}$ if and only if so does the product of signs $\beta(\phi_{\psi'})  \alpha(G,G')$. Indeed, we have the following lemma for classical group $G$.

\begin{lemma}\label{lem MW/W}
    Let $G$ be a pure inner form of classical groups and let $\psi$ be an anti-tempered local Arthur parameter of $G$. Suppose $(G',\psi') \to (\psi,s)$. Then the product of signs 
    \[ e(G)\alpha(G,G') \beta(\phi_{\psi}) \beta(\phi_{\psi'})\]
    only depends on the image of $s$ in $\mathcal{S}_{\psi}$. Moreover, it is a character of $\overline{\mathcal{S}}_{\psi}$ trivial on $s_{\psi}$.
\end{lemma}

We prove this lemma in \S \ref{sec MW/W}. In \S \ref{sec match}, we verify that for symplectic and special orthogonal groups, the character given in Lemma \ref{lem MW/W} matches the character $\varepsilon_{\psi}^{M/MW}$ defined in \cite[Definition 8.1]{Xu17b} for anti-tempered $\psi$. 
From now on, we shall write $\varepsilon_{\psi}^{M/MW}$ for the character in Lemma \ref{lem MW/W}.

Here, we continue the discussion assuming Lemma \ref{lem MW/W}. Let 
\[ \eta_{\psi,x}:=\Tran_{G'}^G (\eta_{\psi'})=  e(G)\varepsilon_{\psi}^{M/MW}(x) \beta(\phi_{\psi}) D_G\circ \Tran_{G'}^G(\eta_{\phi'}),\]
where $x$ is the image of $s$ in $\mathcal{S}_{\psi}$. Expand $\Tran_{G'}^G (\eta_{\phi'})= \eta_{\phi,x}$, we obtain
\begin{align*}
    \eta_{\psi,x}&=e(G) \varepsilon_{\psi}^{M/MW}(x) \beta(\phi_{\psi}) \sum_{\varepsilon \in \widehat{\mathcal{S}}_{\psi,\chi_V}} \varepsilon(x) D_G(\pi(\phi, \varepsilon)) \\
    &=e(G)\sum_{\varepsilon \in \widehat{\mathcal{S}}_{\psi,\chi_V}}  
   \varepsilon\varepsilon_{\psi}^{M/MW}(x) \beta(\phi_{\psi}) \beta(\pi(\phi, \varepsilon))  \widehat{\pi(\phi, \varepsilon)}\\
   &=e(G)\sum_{\varepsilon \in \widehat{\mathcal{S}}_{\psi,\chi_V}}  
   \varepsilon\varepsilon_{\psi}^{M/MW}(s_{\psi} x)  \cdot (\varepsilon(s_{\psi})\beta(\phi_{\psi}) \beta(\pi(\phi, \varepsilon)) ) \cdot\widehat{\pi(\phi, \varepsilon)}\\
   &= e(G) \sum_{\varepsilon \in \widehat{\mathcal{S}}_{\psi,\chi_V}}  
   \varepsilon\varepsilon_{\psi}^{M/MW}(s_{\psi} x)  \left(\gamma(\varepsilon)\widehat{\pi(\phi, \varepsilon)}\right),
\end{align*}
where we set $\gamma(\varepsilon):= \varepsilon(s_{\psi})\beta(\phi_{\psi}) \beta(\pi(\phi, \varepsilon))$ for short. Since $\varepsilon_{\psi}^{MW/W}(-1)=1$ if $-1 \in S_{\psi}$, multiplying by $\varepsilon_{\psi}^{MW/W}$ keeps $\widehat{\mathcal{S}}_{\psi, \chi_V}$ invariant. Thus, we may rewrite
\begin{align*}
    \eta_{\psi,x}
   = e(G) \sum_{\varepsilon \in \widehat{\mathcal{S}}_{\psi,\chi_V}}  
   \varepsilon(s_{\psi} x)  \left(\gamma(\varepsilon \varepsilon_{\psi}^{M/MW})\widehat{\pi(\phi, \varepsilon\varepsilon_{\psi}^{M/MW})}\right).
\end{align*}
Since this holds for any $x \in \mathcal{S}_{\psi}$, comparing with Conjecture \ref{conj 2.2.1}(b-2), 
we obtain 
$$\pi(\psi, \varepsilon\varepsilon_{\psi}^{M/MW})= \gamma(\varepsilon )  \widehat{\pi(\phi, \varepsilon)} $$
as distributions. Therefore, Conjecture \ref{conj 2.2.1}(b-2) holds for $\psi$ if and only if $\gamma(\varepsilon)=1$ for any $\varepsilon \in \widehat{\mathcal{S}}_{\psi,\chi_V}$, which leads us to the following proposition.

\begin{prop}\label{prop eq of sign Moeglin}
    Suppose $\phi$ is a tempered local Arthur parameter of a pure inner form of classical groups. Then for any $\varepsilon \in \widehat{\mathcal{S}}_{\psi,\chi_V}$, we have
    \[\varepsilon(s_{\psi})\beta(\phi_{\psi}) \beta(\pi(\phi, \varepsilon))=1. \]
\end{prop}

For symplectic and special orthogonal groups, this proposition is a special case of \cite[Proposition 4.2]{Moe06b}. In \S \ref{sec eq of sign}, we prove Proposition \ref{prop eq of sign Moeglin} including unitary groups uniformly. Hence, 
we conclude the discussion in this subsection as in the following theorem.

\begin{thm}\label{thm main}
     Conjecture \ref{conj 2.2.1} holds for any anti-tempered local Arthur parameter $\psi$ of any pure inner form of classical group $G(V)$. Moreover, we have the equality of distributions
       \begin{align}\label{eq final mod}
        \varepsilon( s_{\psi}) \pi(\psi, \varepsilon \varepsilon^{M/MW}_{\psi}) = \beta (\phi_{\psi}) \beta(\pi(\phi,\varepsilon)) \widehat{\pi(\phi, \varepsilon)}, 
    \end{align}
    for any $\varepsilon \in \widehat{\mathcal{S}}_{\psi,\chi_V}$
\end{thm}

The last part gives a modification for \cite[Lemma 7.1.1]{Art13} and \cite[Lemma 8.2.2]{Mok15}.

\subsection{Generalization to non-classical groups}\label{sec non-classical gp}
In this subsection, we discuss a possible generalization of the construction of anti-tempered local Arthur packets to non-classical groups $G$.

 For non-classical groups $G$, the analogue of Conjecture \ref{conj 2.2.1} is stated in \cite{Art89}, but Part (a) is not proved yet. However, for tempered local Arthur parameter $\phi$, one may define a stable distribution $\eta_{\phi}$ if the local Langlands correspondence for $G$ is established. For example, for $G=G_2(F)$, this is done in \cite{AX22, GS23}. In particular, \cite[Theorem 10.1.7]{AX22} includes the stability of the distribution associated to each discrete packet. 

Thus, in the following discussion, we assume that there exists a canonical way to assign a stable distribution $\eta_{\phi'}$ for each tempered local Arthur parameter $\phi'$ of an endoscopic group $G'$ of $G$. We also assume that Parts (b-1) and (b-2) of Conjecture \ref{conj 2.2.1} hold for any tempered local Arthur parameter $\phi'$ under the canonical assignment $\phi' \mapsto \eta_{\phi'}$. Then, we construct a candidate for anti-tempered local Arthur packets from these assumptions.

Let $\psi'$ be an anti-tempered local Arthur parameter of an endoscopic group $G'$ of $G$ and put $\phi':= \widehat{\psi'}$, which is tempered. We treat Proposition \ref{prop twisted} as the definition of the stable distribution associated to $\psi'$. That is, we define
\[ \eta_{\psi'}:= \beta(\phi_{\psi'}) D_{G'}(\eta_{\phi'}).\]
Then, if Lemma \ref{lem MW/W} and Proposition \ref{prop eq of sign Moeglin} are verified, the same argument implies that \eqref{eq final mod} in Theorem \ref{thm main} also holds, which gives a candidate for the conjectural local Arthur packet for each anti-tempered local Arthur parameter.

\section{\texorpdfstring{Proof of Lemma \ref{lem MW/W}}{}}\label{sec MW/W}

\subsection{Computation for endoscopic groups}\label{sec computation of endoscopic group}
Suppose $(G', \psi') \to (\psi,s)$. In this subsection, we compute $G'$ and $\psi'$ from each $s \in S_{\psi}$ explicitly.

Recall that we have a decomposition
\begin{align*}
    \psi_{\GL}&= \bigoplus_{i \in I_{gp}} (\rho_i \otimes S_{a_i} \otimes S_{b_i})^{\oplus m_i} + \bigoplus_{i \in I_{bp}} (\rho_i \otimes S_{a_i} \otimes S_{b_i})^{\oplus m_i} \\
    &+\bigoplus_{i \in I_{nsd}} \left(\rho_i |\cdot|^{x_i}\otimes S_{a_i} \otimes S_{b_i} + {}^\sigma \rho_i^{\vee} |\cdot|^{-x_i}\otimes S_{a_i} \otimes S_{b_i} \right)^{\oplus m_i},  
\end{align*}
and
\[  C_{\psi}  \cong \prod_{i \in I_{gp}} \rO_{m_i}(\BC) \times  \prod_{i \in I_{bp}} \Sp_{m_i}(\BC)  \times \prod_{i \in I_{nsd}} \GL_{m_i}(\BC).\]
Let $I:= I_{gp} \sqcup I_{bp} \sqcup I_{nsd}$. 
For any $s \in C_{\psi_{\GL}}$, write $s= (s_i)_{i \in I}$ under the above isomorphism. Let $\Lambda'$ be the set of eigenvalues of $s$ that is not $\pm 1$ and choose $\Lambda \subseteq \Lambda'$ such that for any $\lambda \in \Lambda'$, exactly one of $\lambda,  \lambda^{-1}$ is in $\Lambda$. For each $i \in I$ and $\lambda \in \Lambda \sqcup \{ \pm 1\}$, let $r_{\lambda, i}$ denote the multiplicity of $\lambda$ in the multi-set of eigenvalues of $s_i$. Finally, for $\lambda \in \Lambda \sqcup \{ \pm 1\}$, let
\[ r_{\lambda}:= \sum_{i \in I_{gp} \sqcup I_{bp}} r_{\lambda,i}  \dim( \rho_i|\cdot|^{x_i} \otimes S_{a_i} \otimes S_{b_i} )+\sum_{i \in I_{nsd} } 2r_{\lambda,i} \dim( \rho_i|\cdot|^{x_i} \otimes S_{a_i} \otimes S_{b_i} ) .  \]
Then we have
\[ \widehat{G'} \cong \Cent(s,\widehat{G})^{\circ } \cong  \Aut(M_{1},B_1)^{\circ} \times \Aut(M_{-1}, B_{-1})^{\circ} \times \prod_{\lambda \in \Lambda } \GL_{r_{\lambda}} (\BC) ,  \]
where $M_{\pm 1}$ is the eigenspace of $\pm 1$ of $s$ and $B_{\pm 1}$ is the restriction of $B$ on $M_{\pm 1}$. Note that $\dim(M_{\pm 1})=r_{\pm 1}.$  For $\lambda \in \Lambda$, let
\begin{align*}
    \psi_{\lambda}:&= \bigoplus_{i \in I_{gp}} (\rho_i \otimes S_{a_i} \otimes S_{b_i})^{\oplus r_{\lambda,i}} + \bigoplus_{i \in I_{bp}} (\rho_i \otimes S_{a_i} \otimes S_{b_i})^{\oplus r_{\lambda,i}} \\
    &+\bigoplus_{I_{nsd}} \left(\rho_i |\cdot|^{x_i}\otimes S_{a_i} \otimes S_{b_i} )^{\oplus r_{\lambda ,i}} + {}^\sigma \rho_i^{\vee} |\cdot|^{-x_i}\otimes S_{a_i} \otimes S_{b_i} \right)^{\oplus r_{\lambda^{-1},i}},
\end{align*}
 which is a local Arthur parameter of $\GL_{r_{\lambda}}(E)$. Next, consider the following local Arthur parameter of $\GL_{r_{\pm 1}}(E)$
\begin{align*}
    (\psi_{\pm 1})_{\GL}:&= \bigoplus_{i \in I_{gp}} (\rho_i \otimes S_{a_i} \otimes S_{b_i})^{\oplus r_{\pm 1,i}} + \bigoplus_{i \in I_{bp}} (\rho_i \otimes S_{a_i} \otimes S_{b_i})^{\oplus r_{\pm 1,i}} \\
    &+\bigoplus_{I_{nsd}} \left(\rho_i |\cdot|^{x_i}\otimes S_{a_i} \otimes S_{b_i} + {}^{\sigma }\rho_i^{\vee} |\cdot|^{-x_i}\otimes S_{a_i} \otimes S_{b_i} \right)^{\oplus r_{\pm 1,i}}.  
\end{align*}
Note that $(\psi_{\pm 1})_{\GL}$ may not come from any local Arthur parameter of classical groups. For example, it is possible that $\Aut(M_{1}, B_{ 1})^{\circ} \cong \SO_{2n+1}(\BC)$ but $\det(  (\psi_{ +1})_{\GL})$ is non-trivial. To remedy this, we fix a choice of a (conjugate-)self-dual character $\eta^{\pm}$ of $W_E$ case by case as follows. If $E=F$, let 
\[\eta^{\pm}(w):= \det((\psi_{\pm 1})_{\GL})(w,1,1)\]
if $\Aut(M_{\pm 1}, B_{\pm 1})^{\circ} \cong \SO_{2n+1}(\BC)$, and trivial otherwise. If $E \neq F$, we take $\eta^{\pm}$ to be a conjugate-self-dual character such that for any $s \in W_{F} \setminus W_E$,
\[ \eta^{\pm}(s^2)= (-1)^{\dim_E(V)-r_{\pm 1}}. \]
Then the local Arthur parameter $(\psi^{\pm})_{\GL}:= (\psi_{\pm 1})_{\GL} \otimes \eta^{\pm}$ must come from a local Arthur parameter $\psi^{\pm}$ of some quasi-split group $G(V_{\pm 1})$ with $\widehat{G(V_{\pm })}= \Aut(M_{\pm 1}, B_{\pm 1 })^{\circ}$. We conclude that 
 \begin{align*}
     G' &\cong G(V_+) \times G(V_{-}) \times \prod_{\lambda \in \Lambda} \GL_{r_{\lambda}}(E),\\
     \psi' & \cong \psi^{+} \times \psi^{-}\times  \prod_{\lambda \in \Lambda }\psi_{\lambda}.
 \end{align*}
We remark that different choice of $\eta^{\pm}$ may give rise to different embeddings $\xi: {}^L G' \to {}^L G$, but the endoscopic data are isomorphic. Finally, note that 
\[ \psi_{\GL}= (\psi_{1})_{\GL}+ (\psi_{-1})_{\GL} + \bigoplus_{\lambda \in \Lambda} (\psi_{\lambda} + {}^\sigma \psi_{\lambda}^{\vee} ), \]
and the image of $s$ in $\mathcal{S}_{\psi}$ corresponds to the function $x:I_{gp} \to \{ \pm 1\}$ given by
\begin{align}\label{eq image of s}
    x( \rho_i \otimes S_{a_i} \otimes S_{b_i} )= (-1)^{r_{-1,i}}.
\end{align}

\subsection{Proof of Lemma \ref{lem MW/W}}\label{sec proof of lemma MW/W}
Since $ e(G)\alpha(G, G')= \alpha(G^{\ast}, G')$, we may assume $G=G(V)$ is quasi-split throughout the proof. We shall use the notation developed in the previous subsection.

First, let $\varphi$ be an $L$-parameter of $G=G(V)$. We describe $\beta(\varphi)$ in terms of the decomposition
\begin{align*}
    \varphi_{\GL}= \bigoplus_{i \in I_{gp}} (\rho_i \otimes S_{a_i})^{\oplus m_i} + \bigoplus_{i \in I_{bp}} (\rho_i \otimes S_{a_i} )^{\oplus m_i} +\bigoplus_{i \in I_{nsd}} \left(\rho_i |\cdot|^{x_i}\otimes S_{a_i} + {}^\sigma \rho_i^{\vee} |\cdot|^{-x_i}\otimes S_{a_i}  \right)^{\oplus m_i}. 
\end{align*}
Let $d_{i}:= \dim(\rho_i|\cdot|^{x_i} \otimes S_{a_i})$  for $i \in I=I_{gp}\sqcup I_{bp} \sqcup I_{nsd}$ ($x_i=0$ for $i \in I_{gp} \sqcup I_{bp}$). Then a minimal Levi $M_{\varphi}$ for which the $L$-group contains the image of $\varphi$ is isomorphic to
\[ \prod_{i \in I_{nsd}} (\GL_{d_i}(E))^{\times m_i} \times \prod_{i \in I_{bp}} (\GL_{d_i}(E))^{\times m_i/2} \times \prod_{i \in I_{gp}} (\GL_{d_i}(E))^{\times \lfloor \half{m_i} \rfloor} \times G(V_{an, r}), \]
where (recall that $\mathfrak{r}= \Witt(V)$)
\[ r= \mathfrak{r}-\left( \sum_{i \in I_{gp}}   d_i \cdot \left\lfloor \half{m_i} \right\rfloor +  \sum_{i \in I_{bp}}    \half{d_i m_i} + \sum_{i \in I_{nsd}}   d_im_i  \right).\]
Thus, $\beta(\varphi)= (-1)^{m_{\varphi}+\dim(A_{M_0})}$, where
\[ m_{\varphi}:= \sum_{i \in I_{gp} } \left\lfloor \half{m_i} \right\rfloor+ \sum_{i \in I_{bp}}\half{m_i}+ \sum_{i \in I_{nsd}} m_i.\]

Next, for an $L$-parameter $\varphi_{\GL}$ of $\GL_n(E)$, we compute $\beta_{\GL}(\varphi_{\GL}):=\beta(\varphi_{\GL})$ in terms of the decomposition
\[ \varphi_{\GL}= \bigoplus_{j \in J} (\rho_j|\cdot|^{x_j}\otimes S_{a_j})^{\oplus m_j}. \]
Again, let $d_j= \dim(\rho_j|\cdot|^{x_j}\otimes S_{a_j} )$. In this case, the minimal parabolic subgroup $M_{\varphi}$ is isomorphic to
\[ \prod_{j \in J} (\GL_{d_j}(E))^{\times m_j}, \]
and  $\beta_{\GL}(\varphi_{\GL})= (-1)^{m_{\varphi_{\GL}}+n}$, where 
\[ m_{\varphi_{\GL}}=\sum_{j\in J} m_j. \]
As a consequence, if $\varphi, \varphi_0$ are $L$-parameters of classical groups of the same type and $\varphi_1$ is an $L$-parameter of $\GL_n(E)$ such that $\varphi= \varphi_{0}+ (\varphi_1 + {}^\sigma \varphi_1^{\vee})$, then 
\begin{align}\label{eq beta decomp}
    \beta(\varphi)= \beta(\varphi_0) \beta_{\GL}(\varphi_1).
\end{align}

Now we reduce the computation to the elliptic case. Let $\psi_{ell}:= \psi_1 + \psi_{-1}$, which is a local Arthur parameter of $G_{ell}$, a group of the same type as $G(V)$. Then 
\[ \psi= \psi_{ell} + \bigoplus_{\lambda \in \Lambda} (\psi_{\lambda} + {}^\sigma \psi_{\lambda}^{\vee} ), \]
and hence
\[ \phi_{\psi}= \phi_{\psi_{ell}} + \bigoplus_{\lambda \in \Lambda} (\phi_{\psi_{\lambda}} + {}^\sigma \phi_{\psi_{\lambda}}^{\vee} ).\]
Let $G_{ell}'= G(V^+) \times G(V^{-})$ and $\psi_{ell}'=\psi^+ \times \psi^-$. Then since $G' \cong G_{ell}' \times \prod_{\lambda \in \Lambda} \GL_{r_{\lambda}}(E)$, we have $\alpha(G,G')= \alpha(G_{ell}, G_{ell}')$, which implies (by \eqref{eq beta decomp})
\begin{align*}
    \alpha(G,G') \beta(\phi_{\psi}) \beta(\phi_{\psi'})= \alpha(G_{ell},G_{ell}') \beta(\phi_{\psi_{ell}}) \beta(\phi_{\psi_{ell}'}).
\end{align*}
Therefore, we shall assume $\Lambda$ is empty from now on.

Consider a local Arthur parameter $\psi_{gp}$ of $G_{gp}=G(V_{gp})$, where
\begin{align*}
    \psi_{gp}&= \bigoplus_{i \in I_{gp}} (\rho_i \otimes S_{1} \otimes S_{b_i})^{\oplus m_i},\\
    \psi_{ngp}&= \bigoplus_{i \in I_{bp}} (\rho_i \otimes S_{1} \otimes S_{b_i})^{\oplus \half{m_i}} 
    +\bigoplus_{i \in I_{nsd}} \left(\rho_i \otimes S_{1} \otimes S_{b_i} \right)^{\oplus m_i},
\end{align*}
so that $\psi= \psi_{gp}+ (\psi_{ngp}+ {}^\sigma \psi_{ngp}^{\vee}).$ Let $ I_{gp,\rho}:=\{i \in I_{gp}\ | \ \rho_i \cong \rho\}$ and rewrite
\[ \psi_{gp}=  \bigoplus_{\rho}\bigoplus_{i \in I_{gp,\rho}} (\rho \otimes S_{1} \otimes S_{b_i})^{\oplus m_i}. \]
Let $R_0$ (resp. $R_1$) be the collection of $\rho$ such that $I_{gp,\rho} $ is non-empty and $b_i$ is even (resp. odd) for any $i \in I_{gp, \rho}$. Finally, let $m_{\rho}= \sum_{i \in I_{\rho, gp}} m_i$. We have
\[ \phi_{\psi}= \phi_0+ (\phi_{1}+ {}^\sigma \phi_1^{\vee})+ (\phi_{\psi_{ngp}} +{}^\sigma \phi_{\psi_{ngp}}^{\vee}),\]
where
\begin{align*}
    \phi_1= \left(\bigoplus_{\rho \in R_0} \bigoplus_{i \in I_{gp, \rho}} \left( \bigoplus_{k=0}^{(b_i-2)/2} \rho|\cdot|^{\half{1}+k} \right)^{\oplus m_i}+ \bigoplus_{\rho \in R_1} \left( \rho^{\oplus \left\lfloor \half{m_{\rho}} \right\rfloor} +\bigoplus_{i \in I_{gp, \rho}} \left(  \bigoplus_{k=0}^{(b_i-3)/2} \rho|\cdot|^{1+k} \right)^{\oplus m_i}\right)\right)\otimes S_1,
\end{align*}
and 
\begin{align*}
\phi_{0}=\left(\bigoplus_{\rho \in R_1}  \rho^{ \oplus m_{\rho}- 2 \left\lfloor \half{m_{\rho}} \right\rfloor} \right)\otimes S_1.
\end{align*}
Thus, let 
\[ f_{\rho}:= \begin{cases}
    \sum_{i \in I_{gp, \rho}}\half{m_ib_i} & \text{ if }\rho \in R_0,\\
   \left\lfloor \half{m_{\rho}} \right\rfloor + \sum_{i \in I_{gp, \rho}} \half{m_i(b_i-1)} & \text{ if }\rho \in R_1.
\end{cases}\]
We obtain that $M_{\phi_{\psi}}$ is isomorphic to
\[ M_{\phi_{\psi_{ngp}}} \times \prod_{\rho \in  R_0 \sqcup R_1} (\GL_{\dim(\rho)}(E))^{\times f_{\rho}} \times G(V'),\]
where $V\cong V' \oplus \H^{\dim(\psi_{ngp})+\sum_{\rho \in R_0 \sqcup R_1}f_{\rho}}$. This implies 
\begin{align}\label{eq beta(phi_psi)}
    \beta(\phi_{\psi})= \beta_{\GL}(\phi_{\psi_{ngp}}) (-1)^{f} (-1)^{r(G_{gp})},
\end{align}
where
\[ f=  \sum_{\rho \in R_0 \sqcup R_1} f_{\rho}.\]

The same computation works for $ \beta(\phi_{\psi_{\pm 1}})$ by replacing $m_i$ with $r_{\pm 1 , i}$. To be explicit, let $\psi_{gp}^{\pm}$ be a local Arthur parameter of $G_{gp}^{\pm}=G(V_{gp}^{\pm})$, where
\begin{align*}
    \psi_{gp}^{\pm}&:= \bigoplus_{i \in I_{gp}} (\rho_i \otimes S_{1} \otimes S_{b_i})^{\oplus r_{\pm 1,i}},\\
    \psi_{ngp}^{\pm}&:= \bigoplus_{i \in I_{bp}} (\rho_i \otimes S_{1} \otimes S_{b_i})^{\oplus \half{r_{\pm 1,i}}} 
    +\bigoplus_{i \in I_{nsd}} \left(\rho_i \otimes S_{1} \otimes S_{b_i} \right)^{\oplus r_{\pm 1, i}},
\end{align*}
so that $\psi^{\pm}= \psi_{gp}^{\pm}+ (\psi_{ngp}^{\pm}+ {}^\sigma (\psi_{ngp}^{\pm})^{\vee})$. Let $m_{\rho}^{\pm}:= \sum_{i \in I_{gp, \rho}} r_{\pm 1, i}$ and 
\[ f_{\rho}^{\pm}:= \begin{cases}
    \sum_{i \in I_{gp, \rho}}\half{r_{\pm 1, i}b_i} & \text{ if }\rho \in R_0,\\
   \left\lfloor \half{m_{\rho}^{\pm}} \right\rfloor + \sum_{i \in I_{gp, \rho}} \half{r_{\pm 1, i}(b_i-1)} & \text{ if }\rho \in R_1.
\end{cases}\]
Then $\beta(\phi_{\psi^{\pm}})= \beta_{\GL}(\phi_{\psi_{ngp}^{\pm}}) (-1)^{f^{\pm}} (-1)^{r(G_{gp}^{\pm})}$, where
\[ f^{\pm}= \sum_{\rho \in R_0 \sqcup R_1} f_{\rho}^{\pm}.\]

Finally, we compute $\alpha(G,G')\beta(\phi_{\psi})\beta(\phi_{\psi'})$. Note that $\psi_{ngp}= \psi_{ngp}^+ + \psi_{ngp}^-$ and hence 
\begin{enumerate}
    \item [$\oldbullet$] $\beta_{\GL}(\phi_{\psi_{ngp}})= \beta_{\GL}(\phi_{\psi_{ngp}}^+) \beta_{\GL}(\phi_{\psi_{ngp}^-}),$ and 
    \item [$\oldbullet$] $\alpha(G, G')= (-1)^{r(G(V))+ r(G(V_1))+ r(G(V_{-1}))}= (-1)^{ r(G_{gp})+ r(G_{gp}^{+})+ r(G_{gp}^{-})}$.
\end{enumerate}
 Since $m_i= r_{+1, i}+ r_{-1, i}$, we obtain
\begin{align*}
\alpha(G,G') \beta(\phi_{\psi}) \beta(\phi_{\psi'})&=  \alpha(G,G') \beta(\phi_{\psi})\beta(\phi_{\psi^+})\beta(\phi_{\psi^-})\\ 
&= \prod_{\rho \in R_1} (-1)^{\left\lfloor \half{m_{\rho}^{+}} \right\rfloor+ \left\lfloor \half{m_{\rho}^-} \right\rfloor+ \left\lfloor \half{m_{\rho}} \right\rfloor }\\
&= \prod_{\rho \in R_1 }  (-1)^{m_{\rho}^{-}   (m_{\rho}-1)}\\
&=\prod_{\rho \in R_1 } \prod_{i \in I_{gp,\rho}} (-1)^{ r_{-1, i}  (m_{\rho}-1)}\\
&=\prod_{i \in I_{gp}} (-1)^{r_{-1, i} b_i   (m_{\rho_i}-1)}.
\end{align*}
The third equality follows from the observation that since $m_{\rho}= m_{\rho}^++ m_{\rho}^{-}$, 
\begin{align*}
    (-1)^{\left\lfloor \half{m_{\rho}^{+}} \right\rfloor+ \left\lfloor \half{m_{\rho}^-} \right\rfloor+ \left\lfloor \half{m_{\rho}} \right\rfloor }= \begin{cases}
        1 & \text{ if }m_{\rho} \text{ is odd,}\\
        (-1)^{m_{\rho}^-}  & \text{ if }m_{\rho} \text{ is even.}
    \end{cases}
\end{align*}
Let $x$ be the image of $s$ in $\mathcal{S}_{\psi}$ described by \eqref{eq image of s}. From the above computation, we conclude that 
\[ \alpha(G,G') \beta(\phi_{\psi}) \beta(\phi_{\psi'})= \langle \varepsilon, x \rangle, \]
where $\widehat{\mathcal{S}}_{\psi} \ni \varepsilon: I_{gp} \to \{\pm 1\}$ is given by
\begin{align}\label{eq epsilon}
    \varepsilon(\rho_i \otimes S_{1} \otimes S_{b_i})= (-1)^{ b_i (m_{\rho_i }-1)}.
\end{align}
It is straightforward to check that
\[ \langle \varepsilon ,e_0 \rangle= \prod_{i \in I_{gp}} (-1)^{b_i m_i (m_{\rho_i}-1)}=1,\]
and hence $ \varepsilon \in \widehat{\overline{\mathcal{S}}}_{\psi}.$ Also, by \eqref{eq eva s_psi},
\begin{align*}
    \varepsilon(s_{\psi})= \prod_{i \in I_{gp}} \varepsilon( \rho_i \otimes S_{a_i} \otimes S_{b_i})^{ m_i (b_i-1)}=\prod_{i \in I_{gp}} (-1)^{  (m_{\rho_i}-1) b_i(b_i-1)}=1.
\end{align*}
This completes the proof of the lemma. \qed

\subsection{\texorpdfstring{Match with the character $\varepsilon_{\psi}^{M/MW}$}{}}\label{sec match}
In this subsection, we match the character $\varepsilon$ derived in the previous subsection with the character $\varepsilon_{\psi}^{M/MW}$ defined in \cite[Definition 8.1]{Xu17b} for anti-tempered $\psi$. 

Without loss of generality, we assume $\psi$ is of good parity, i.e. $I_{bp} \sqcup I_{nsd} = \emptyset$, and write
\begin{align*}
    \psi_{\GL}&= \bigoplus_{i \in I_{gp}} (\rho_i \otimes S_{1} \otimes S_{b_i})^{\oplus m_i}\\
    &= \bigoplus_{\rho} \bigoplus_{i \in I_{gp, \rho}} (\rho\otimes S_{1} \otimes S_{b_i})^{\oplus m_i}\\
    &= \bigoplus_{\rho} \bigoplus_{i \in I_{\rho}} \rho \otimes S_{1} \otimes S_{b_i}.
\end{align*}
Here $I_{\rho}$ is another index set. To match with Xu's notation, we need to fix a collection of signs $\zeta=(\zeta_i)_{i \in I_{\rho}}$ such that $\zeta_i (a_i- b_i ) \geq 0$ and a total order $>$ on $I_{\rho}$ such that the following property holds.
\begin{enumerate}
    \item [(P)] If $A_i>A_j$, $B_i>B_j$ and $\zeta_i=\zeta_j$, then $i>j$,
\end{enumerate}
where
\[ A_i:= \frac{a_i+b_i }{2}-1,\ B_1:=  \half{\zeta_i(a_i-b_i)}. \]
We fix $\zeta:= (-1)_{i \in I_{\rho}}$ and $>$ any total order on $I_{\rho}$ that is non-decreasing with respect to $b_i$. One can directly check from \eqref{eq epsilon} that $\varepsilon= \varepsilon_{\psi}^{M/MW}$ defined below.

\begin{defn}[{\cite[Definitions 8.1]{Xu17b}}]

Let $\psi$ be a local Arthur parameter of quasi-split symplectic or special orthogonal groups and fix a collection of signs $\zeta$ and a total order $>$ on each $I_{\rho}$ satisfying (P). Define a  character $\varepsilon^{M/MW}_{\psi}$ as follows,
\begin{align*}
    \varepsilon^{M/MW}_{\psi}(\rho \otimes S_{a_i} \otimes S_{b_i}):= \begin{cases}
        1 & \textrm{ if }a_i+b_i \textrm{ is odd,}\\
        1 & \textrm{ if }a_i, b_i \textrm{ are both even},\\
        (-1)^{m} & \textrm{ if }a_i, b_i \textrm{ are both odd and } \zeta_i=+1,\\
         (-1)^{m+n} & \textrm{ if }a_i, b_i \textrm{ are both odd and } \zeta_i=-1,
    \end{cases}
\end{align*}
    where the number $m$ and $n$ are defined by
    \begin{align*}
        m&=\#\{ j \in I_{\rho}\ | \ j >i,\ a_j, b_j \textrm{ are both odd, and }\zeta_j=-1\},\\
        n &=\#\{ j \in I_{\rho}\ | \ j <i,\ a_j, b_j \textrm{ are both odd}\}.
    \end{align*}
\end{defn}

From now on, we shall write $\varepsilon_{\psi}^{M/MW}$ for the character $\varepsilon$.

\begin{remark}\label{rmk in Lpacket}
Let $\psi$ be a local Arthur parameter of $G$. One can associate the group $S_{\phi_{\psi}}:= \textrm{Cent}(\textrm{Im}(\phi_{\psi}, \widehat{G}))$ and the component group $\mathcal{S}_{\phi_{\psi}}:= S_{\phi_{\psi}}/ S_{\phi_{\psi}}^{\circ}$ similarly. Since $\textrm{im}(\phi_{\psi}) \subseteq \textrm{im}(\psi)$, the injection $S_{\psi} \hookrightarrow S_{\phi_{\psi}}$ induces a surjection  $\mathcal{S}_{\psi} \twoheadrightarrow \mathcal{S}_{\phi_{\psi}}$ and an injection $\widehat{\mathcal{S}}_{\phi_{\psi}} \hookrightarrow \widehat{\mathcal{S}}_{\psi}.$ We describe these maps explicitly, and show that $\varepsilon_{\psi}^{M/MW}$ is in the image of $\widehat{\mathcal{S}}_{\phi_{\psi}}$ when $\psi$ is anti-tempered.

The combinatorial descriptions of the component group $\mathcal{S}_{\phi_{\psi}}$ and its characters $\widehat{\mathcal{S}}_{\phi_{\psi}}$ are similar to the ones given in \S \ref{sec component group} for $ \mathcal{S}_{\psi}$ and $\widehat{\mathcal{S}}_{\psi}$. We state it in terms of the decomposition of $\psi$ now. For simplicity, we assume $\psi$ is of good parity, and write
\[ \psi_{\GL}= \bigoplus_{i \in I_{gp}} (\rho_i\otimes S_{a_i} \otimes S_{b_i})^{\oplus m_i}. \]
Let $I_{gp}^{o}= \{ i \in I_{gp}\ | \ b_i \textrm{ is odd.}\}$. Let $\mathcal{C}_{\phi_{\psi}}$ be the set of functions $e:I_{gp}^o \to \{ \pm 1\}$. Again we write $e(\rho_i \otimes S_{a_i} \otimes S_{b_i})$ instead of $e(i)$. Consider the map $\det: \mathcal{C}_{\phi_{\psi}} \to \{ \pm 1\}$ given by
\[ \det (e):= \prod_{i \in I_{gp}^o} e(\rho_i \otimes S_{a_i} \otimes S_{b_i})^{\dim(\rho_i \otimes S_{b_i})} .\]
Let $\mathcal{C}^{+}_{\phi_{\psi}}= \{e \in \mathcal{C}_{\phi_{\psi}}\ | \ \det(e)=1\}$. Then $\mathcal{S}_{\phi_{\psi}} \cong \mathcal{C}_{\phi_{\psi}}$ if $G=\RU(V)$ and $\mathcal{S}_{\phi_{\psi}} \cong \mathcal{C}_{\phi_{\psi}}^{+}$ otherwise.

Recall that we also regard $\mathcal{S}_{\psi}$ as the set of functions $e: I_{gp} \to \{\pm 1\}$. The surjection $\mathcal{S}_{\psi} \twoheadrightarrow \mathcal{S}_{\phi_{\psi}}$ is exactly the restriction $e \mapsto e|_{I_{gp}^o}$. The injection $\widehat{\mathcal{S}}_{\phi_{\psi}} \hookrightarrow \widehat{\mathcal{S}}_{\psi}$ identifies $ \widehat{\mathcal{S}}_{\phi_{\psi}}$ with the subgroup of functions
\[ \{ \varepsilon: I_{gp} \to \{\pm 1\}\ | \ \varepsilon(\rho_i \otimes S_{a_i} \otimes S_{b_i})=1 \ \forall i \in I_{gp} \setminus I_{gp}^o \}.\]
Under this identification, we conclude that $\varepsilon_{\psi}^{M/MW} \in \widehat{\mathcal{S}}_{\phi_{\psi}}$ when $\psi$ is anti-tempered.    
\end{remark}

\section{\texorpdfstring{Proof of Proposition \ref{prop eq of sign Moeglin} }{}}\label{sec eq of sign}

\subsection{Preparations}\label{sec prep inj}
In this subsection, we collect the statements we need in the reduction process in the proof. Let $\phi$ be a tempered local Arthur parameter of $G$ and write
\begin{align*}
    \phi_{\GL}&= \bigoplus_{i \in I_{gp}} (\rho_i \otimes S_{a_i} \otimes S_{1})^{\oplus m_i} + \bigoplus_{i \in I_{bp}} (\rho_i \otimes S_{a_i} \otimes S_{1})^{\oplus m_i} \\
    &+\bigoplus_{i \in I_{nsd}} \left(\rho_i\otimes S_{a_i} \otimes S_{1} +{}^\sigma \rho_i^{\vee}\otimes S_{a_i} \otimes S_{1} \right)^{\oplus m_i},  
\end{align*}
We say $\phi$ is discrete if $I_{bp} \sqcup I_{nsd}= \emptyset$ and $m_i=1$ for all $i \in I_{gp}$. The following proposition allows us to reduce from general case to the discrete case.

\begin{prop}\label{prop reduction to discrete}
Let $\phi$ be a tempered local Arthur parameter of $G=G(V)$. Suppose $\phi= \phi_0 + (\phi_1 + {}^\sigma \phi_1^{\vee}) $, where $\phi_0$ is a tempered local Arthur parameter of $G(V_0)$ and $\phi_1$ is a tempered local Arthur parameter of $\GL_n(E)$. Then there exists an injection $\mathcal{S}_{\phi_0} \hookrightarrow \mathcal{S}_{\phi}$. Let $\tau_1$ be the unique element in $\Pi_{\phi_1}$. Then for any $\varepsilon_0 \in \widehat{\mathcal{S}}_{\phi_0, \chi_{V_0}}$, we have
\[ \tau_1 \rtimes \pi(\phi_0, \varepsilon_0) = \bigoplus_{ \substack{\varepsilon \in \widehat{\mathcal{S}}_{\psi, \chi_V},\\  \varepsilon|_{\mathcal{S}_{\phi_0}}=\varepsilon_0}} \pi(\phi, \varepsilon). \]
In other words, for any $\varepsilon \in \widehat{\mathcal{S}}_{\phi,\chi_V}$, we have 
\[ \pi(\phi, \varepsilon) \hookrightarrow \tau_1 \rtimes \pi( \phi_0 , \varepsilon|_{\mathcal{S}_{\psi}}) \]
\end{prop}

\begin{proof}
For quasi-split symplectic and orthogonal groups, see the proof of \cite[Proposition 2.4.3]{Art13}. For quasi-split unitary groups, see the proof of \cite[Proposition 3.4.4]{Mok15}. For non-quasi-split classical groups, see the proof of \cite[Proposition 8.3.6]{MR18} and  \cite[\S 4.7]{KMSW14}.   
\end{proof}

For a discrete local Arthur parameter $\phi$, we rewrite the decomposition as 
\[ \phi_{\GL}= \bigoplus_{\rho} \bigoplus_{i \in I_{\rho}} \rho \otimes  S_{a_i} \otimes S_1.\]
We need the following description for the cuspidality of $\pi(\phi,\varepsilon)$.

\begin{thm}[{\cite[Theorem 2.5.1]{Moe11},\cite[Corollary 3.5]{MR18}}]\label{thm moeglin supercuspidal}
Let $\phi$ be a tempered local Arthur parameter of $G=G(V)$. The representation $\pi(\phi,\varepsilon)$ is supercuspidal if and only if the followings hold.
    \begin{enumerate}  
    \item [(a)] The parameter $\phi$ is discrete.
        \item [(b)] If $i \in I_{\rho}$ and $a_i \geq 3$, then there exists a $j \in I_{\rho}$ such that $a_j= a_i-2$ and 
\[\varepsilon(\rho\otimes S_{a_i} \otimes S_1)\varepsilon(\rho \otimes S_{a_j} \otimes S_1)=-1.\]
        \item [(c)] If $i \in I_{\rho}$ and $a_i=2$, then $\varepsilon(\rho \otimes S_{a_i}\otimes S_1)=-1$.
    \end{enumerate}
    We say $\varepsilon$ is alternating if Conditions (b) and (c) hold.
\end{thm}

 Let $\phi$ be a tempered local Arthur parameter of $G=G(V)$ and $\varepsilon \in \widehat{\mathcal{S}}_{\phi, \chi_V}$. Suppose $\pi(\phi, \varepsilon)$ is not supercuspidal. We associate a representation $\pi(\phi^{-}, \varepsilon^{-})$ of $G(V^{-})$ of smaller rank as follows.

\begin{defn}\label{def reduction of tempered}
     Suppose $\phi$ is a tempered $L$-parameter of $G=G(V)$ and $\varepsilon \in \widehat{\mathcal{S}}_{\phi, \chi_V}$. 
    \begin{enumerate}
    \item [(a)] Suppose $\phi$ is not discrete. Then $\phi= \phi_0+ (\phi_1 + {}^{\sigma}\phi_1^{\vee})$ for some discrete $\phi_0$ of smaller rank. Let $\phi^{-}:= \phi_0$ and $\varepsilon^{-}:= \varepsilon|_{\mathcal{S}_{\phi_0}}$.
    \item [(b-1)] Suppose $\phi$ is discrete and there exists an $i \in I_{\rho}$ such that $a_i \geq 3$ but $a_j \neq a_i-2$ for any $j \in I_{\rho}$. Then let $\phi^{-}:= \phi- \rho \otimes S_{a_i} \otimes S_{1}+ \rho \otimes S_{a_i-2} \otimes S_{1}$, a tempered local Arthur parameter of $G(V^{-})$ of smaller rank, and let $\varepsilon^-$ be the image of $\varepsilon$ under the natural isomorphism $\widehat{\mathcal{S}}_{\phi,\chi_V} \cong \widehat{\mathcal{S}}_{\phi^-,\chi_{V^-}}$.
    \item [(b-2)] Suppose $\phi$ is discrete and there exist $i, j \in I_{\rho}$ such that $a_j= a_i-2$ but 
    \[\varepsilon(\rho\otimes S_{a_i} \otimes S_1)\varepsilon(\rho \otimes S_{a_j} \otimes S_1)=1.\]
    Let $\phi^{-}:= \phi- \rho \otimes S_{a_i} \otimes S_{1}+ \rho \otimes S_{a_i-2} \otimes S_{1}$, a tempered local Arthur parameter of $G(V^{-})$ of smaller rank, and let $\varepsilon^{-}:= \varepsilon|_{\mathcal{S}_{\phi^-}}$.
    \item [(c)] Suppose $\phi$ is discrete and there exists an $i \in I_{\rho}$ such that $a_i=2$ but $\varepsilon(\rho \otimes S_{a_i} \otimes S_1)=1$. Then let $\phi^{-}:= \phi- \rho \otimes S_{a_i} \otimes S_{1}$, a tempered local Arthur parameter of $G(V^{-})$ of smaller rank, and let $\varepsilon^{-}:= \varepsilon|_{\mathcal{S}_{\phi^-}}$.
\end{enumerate}
\end{defn}

Each $n$-dimensional irreducible (conjugate-)self-dual representation $\rho$ of $W_E$ corresponds to a (conjugate-)self-dual supercuspidal representation of $\GL_n(E)$ by local Langlands correspondence of $\GL_n(E)$. By abuse of notation, we also denote this supercuspidal representation by $\rho$.

\begin{lemma}\label{lem Jacquet module}
For the representations $\pi(\phi,\varepsilon)$ and $\pi(\phi^{-},\varepsilon^{-})$ given in Definition \ref{def reduction of tempered}, we have the following injection in each case.
\begin{enumerate}
    \item [(a)] $\pi(\phi, \varepsilon) \hookrightarrow \tau_{\phi_1} \rtimes \pi(\phi^{-}, \varepsilon^{-})$, where $\tau_{\phi_1}$ is the unique element in the tempered local Arthur packet $\Pi_{\phi_1}$ of $\GL_n(E)$.
    \item [(b-1)] $\pi(\phi, \varepsilon) \hookrightarrow \rho|\cdot|^{\half{a_i-1}} \rtimes \pi(\phi^{-}, \varepsilon^{-})$.
    \item [(b-2)] $\pi(\phi, \varepsilon) \hookrightarrow \rho|\cdot|^{\half{a_i-1}} \times \rho|\cdot|^{\half{a_i-3}}  \times  \cdots \times \rho|\cdot|^{\half{-a_i+3}} \rtimes \pi(\phi^{-}, \varepsilon^{-})$.
    \item [(c)] $\pi(\phi, \varepsilon) \hookrightarrow \rho|\cdot|^{\half{1}} \rtimes \pi(\phi^{-}, \varepsilon^{-})$.
\end{enumerate}
\end{lemma}

\begin{proof}
Part (a) follows from Proposition \ref{prop reduction to discrete}. The rest of the assertion follows from \cite[Theorem 7.1]{AM23} for symplectic groups and split special odd orthogonal groups, and from the proof of \cite[Theorem 8.3.4]{MR18} for unitary and special orthogonal groups.
\end{proof}

\subsection{Proof of Proposition \ref{prop eq of sign Moeglin}}\label{sec proof of prop eq of sign Moeglin}
First, we compute the product of sign $\varepsilon(s_{\psi})\beta(\phi_{\psi}) \beta(\pi(\phi, \varepsilon))$ when $\pi(\phi, \varepsilon)$ is a supercuspidal representation of $G=G(V)$. By Theorem \ref{thm moeglin supercuspidal}, we may write
\begin{align}\label{eq supercuspidal L-par}
\begin{split}
    \phi_{\GL}=   \bigoplus_{\rho \in R_{0}} \rho \otimes ( S_{2} + S_{4} + \cdots + S_{2m_{\rho}})  \otimes S_1 
     + \bigoplus_{\rho \in R_{1}} \rho \otimes ( S_{1} + S_{3} +\cdots + S_{2m_{\rho}-1}) \otimes S_1,
    \end{split}
\end{align}
and $\varepsilon$ is alternating. The decomposition of $\psi_{\GL}$ can be obtained from \eqref{eq supercuspidal L-par} by replacing each $\rho \otimes S_{a} \otimes S_1$ with $\rho \otimes S_{1} \otimes S_a$.  By abuse of notation, we let $\varepsilon(\rho\otimes S_{1}\otimes S_{a}):=\varepsilon(\rho\otimes S_{a}\otimes S_{1})$.

First, we compute $\varepsilon(s_{\psi})$. By \eqref{eq eva s_psi} and the condition that $\varepsilon$ is alternating, we have
\begin{align*}
    \varepsilon(s_{\psi})&= \prod_{\rho \in R_0} \prod_{k=1}^{m_{\rho}} \varepsilon(\rho\otimes S_{1}\otimes S_{2k})^{2k-1} \times \prod_{\rho \in R_1} \prod_{k=1}^{m_{\rho}} \varepsilon(\rho\otimes S_{1}\otimes S_{2k-1})^{2k}\\
    &= \prod_{\rho \in R_0} \prod_{k=1}^{m_{\rho}} (-1)^k\\
    &= \prod_{\rho \in R_0}(-1)^{\half{m_\rho (m_{\rho}+1)}}.
\end{align*}

Next, we compute $\beta(\phi_{\psi})$. By the computation in \S \ref{sec proof of lemma MW/W}, we have $\beta(\phi_{\psi})= (-1)^f (-1)^{r(G)}$, where
\begin{align*}
    f&= \sum_{\rho \in R_{0}} \half{m_\rho (m_{\rho}+1)} + \sum_{\rho \in R_1} \left(\left\lfloor \half{m_{\rho}} \right\rfloor + \half{m_{\rho}( m_{\rho}-1)} \right)\\
    & \equiv \sum_{\rho \in R_{0}} \half{m_\rho (m_{\rho}+1)} \mod 2.
\end{align*}

Finally, since $\pi(\phi,\varepsilon)$ is supercuspidal, it follows from the definition that $\beta(\pi(\phi,\varepsilon))= (-1)^{r(G)}$. Therefore,
we have verified $\varepsilon(s_{\psi})\beta(\phi_{\psi}) \beta(\pi(\phi, \varepsilon))=1$ in this case.

Next, it suffices to verify that for the pair of representations $\pi(\phi, \varepsilon)$ and $\pi(\phi^{-}, \varepsilon^{-})$ defined in Definition \ref{def reduction of tempered}, the following equality holds 
\begin{align}\label{eq goal product of sign}
    \varepsilon(s_{\psi})\beta(\phi_{\psi}) \beta(\pi(\phi, \varepsilon))=\varepsilon^-(s_{\psi^{-}})\beta(\phi_{\psi^-}) \beta(\pi(\phi^{-}, \varepsilon^{-})),
\end{align}
where $\psi^{-}:= \widehat{\phi^{-}}$. We verify the above equation case by case. Before starting the verification, we give three useful observations.
\begin{enumerate}
    \item [(i)] If $\varphi= \varphi_0+ (\varphi_1+ {}^{\sigma}\varphi_1^{\vee})$, where $\varphi, \varphi_{0}$ are $L$-parameters of classical groups of the same type and $\varphi_1$ is an $L$-parameter of $\GL_n(E)$, then 
    \[ \beta(\varphi)= \beta (\varphi_0) \beta_{\GL}(\varphi_1).\]
    \item [(ii)] If $\pi \hookrightarrow \tau \rtimes \pi_0$, then $\beta(\pi)= \beta(\pi_0) \beta_{\GL}(\tau)$.
    \item [(iii)] If $\varphi$ is a tempered local Arthur parameter of $\GL_n(E)$ and $\tau \in \Pi_{\varphi}$, then $\beta_{\GL}(\tau)= \beta_{\GL}(\phi_{\widehat{\varphi}} ). $
\end{enumerate}
Observation (i) is computed in \eqref{eq beta decomp}. Observation (ii) follows from definition. For Observation (iii), suppose $\tau$ is an irreducible representation of $\GL_n(E)$ with $L$-parameter $\varphi_{\tau}$. Then the supercuspidal support of $\tau$ is characterized by $\lambda_{\tau}: W_E \to \GL_n(\BC)$ defined by
\[ \lambda_{\tau}(w):= \phi_{\tau}\left(w, \begin{pmatrix}
    |w|^{1/2} & \\ & |w|^{-1/2}
\end{pmatrix}  \right). \]
In particular, $\beta_{\GL}(\tau)= \beta_{\GL}(\lambda_{\tau}\otimes S_1)$. One can check that $\phi_{\widehat{\varphi}}= \lambda_{\tau}$ in the setting of Observation (iii).

Now we start the verification of \eqref{eq goal product of sign}. For Case (a), we have $\psi= \psi^-+ \widehat{\phi_1}+ \widehat{{}^\sigma\phi_1^{\vee}}$. Thus $\varepsilon(s_{\psi})=\varepsilon^-(s_{\psi^-})$ by \eqref{eq eva s_psi} and $\beta(\phi_{\psi})= \beta(\phi_{\psi^-}) \beta_{\GL}( \phi_{\widehat{\phi_1}})$ by Observation (i). By Observation (iii), to verify \eqref{eq goal product of sign}, it remains to check that 
\[ \beta(\pi(\phi, \varepsilon))= \beta_{\GL}(\tau_{\phi_1}) \beta(\pi(\phi^{-}, \varepsilon^{-})).\]
However, this is a direct consequence of Lemma \ref{lem Jacquet module}(a) and Observation (ii). This completes the verification of this case.

For Case (b-1), we have $\psi^{-}= \psi- \rho\otimes S_{1}\otimes S_{a_{i}}+ \rho\otimes S_{1}\otimes S_{a_{i}-2}$, and hence $\varepsilon(s_{\psi})=\varepsilon^{-}(s_{\psi^-})$ and
\[ \phi_{\psi}= \phi_{\psi^-}+ \rho|\cdot|^{\half{a_i-1}}\otimes S_1 + \rho|\cdot|^{\half{-a_i+1}}\otimes S_1. \]
Thus \eqref{eq goal product of sign} is again a consequence of the observations and \ref{lem Jacquet module}(b-1). This completes the verification of this case.

For Case (b-2), we have $\psi= \psi^- + \rho \otimes S_1 \otimes S_{a_i}+ \rho \otimes S_1 \otimes S_{a_i-2}$, and hence $\varepsilon(s_{\psi})=\varepsilon^{-}(s_{\psi^-})$ and 
\[ \phi_{\psi}= \phi_{\psi^-} + \left(\bigoplus_{k=0}^{a_i-2} \rho|\cdot|^{\half{a_i-1}-k}\otimes S_1\right)+ {}^\sigma\left(\bigoplus_{k=0}^{a_i-2} \rho|\cdot|^{\half{a_i-1}-k}\otimes S_1\right)^{\vee}. \]
Thus, $\beta(\phi_{\psi})= \beta(\phi_{\psi^{-}}) (-1)^{ (a_i-1)(\dim(\rho)-1)}$. On the other hand, Lemma \ref{lem Jacquet module}(b-2) also implies 
\[ \beta(\pi(\phi,\varepsilon))= \beta( \pi(\phi^{-},\varepsilon^{-})) (-1)^{ (a_i-1)(\dim(\rho)-1)}.\]
This completes the verification of this case.

For Case (c), we have $\psi= \psi^{-}+ \rho \otimes S_1 \otimes S_2.$ Since $\varepsilon(\rho \otimes S_2 \otimes S_1)=1$, we have $\varepsilon(s_{\psi})= \varepsilon^-(s_{\psi^{-}})$ by \eqref{eq eva s_psi}. Also,
\[ \phi_{\psi}= \phi_{\psi^-}+ \rho|\cdot|^{\half{1}}\otimes S_1 + \rho|\cdot|^{\half{-1}}\otimes S_1. \]
Thus \eqref{eq goal product of sign} is again a consequence of the observations and Lemma \ref{lem Jacquet module}(c). This completes the verification of this case and the proof of the proposition. \qed

\section{\texorpdfstring{Proof of Proposition \ref{prop twisted}}{}}\label{sec working hypothesis}
We prove Proposition \ref{prop twisted} in this section. The argument is based on \cite[\S 3]{MW06}, \cite[\S 6.3, Appendix]{Xu17b} and the explicit computation of $\beta(\phi_{\psi})$ (see \eqref{eq beta(phi_psi)}), which may be known to expert. We provide enough background and explain the difference between our setting and \cite{MW06}.

\subsection{\texorpdfstring{Disconnected $\GL_N^+(E)$}{}}\label{sec twisted endoscopic transfer}
We recall the setting for twisted endoscopic transfer and the disconnected $\GL_N$ involved in this subsection. Let $G$ be the classical groups in consideration. We realize $G$ as an elliptic twisted endoscopic group of $\GL_N^+(E)= \GL_N(E) \rtimes \langle \theta \rangle$. Here $\theta$ is an outer automorphism of $\mathrm{Res}_{E/F}\GL_{N/E}(F)=\GL_N(E)$ of order 2 that preserves an $F$-splitting of $\mathrm{Res}_{E/F}\GL_{N/E}$. Let $\GL_N^{\theta}(E):= \GL_N(E) \rtimes \theta$. We fix a Whittaker datum of $\GL_N(E)$ fixed by $\theta$ throughout this section. See \cite[\S 2.1]{Art13} and \cite[\S 2.4]{Mok15} for the precise setting. 

Here is what we need about $\theta$ in the explicit computation later. We identify the set of (restricted) simple roots $\Delta$ of $\GL_N(E)$ as $\{\alpha_1,\ldots, \alpha_{N-1}\}$, and $\theta(\alpha_i)=\alpha_{N-i}$. For any subset $I \subseteq \Delta$, let $M_{I}$ denote the corresponding standard Levi subgroup that contains $I$ as simple roots. Any Levi $M$ of $\GL_N(E)$ fixed by $\theta$ is of the form
\[ \GL_{n_1}(E) \times \cdots \times \GL_{n_s}(E) \times \GL_{n_0}(E) \times \GL_{n_s}(E) \times \cdots \times \GL_{n_1}(E). \]
Let $m=(m_1,\ldots, m_s, m_0 ,m_s', \ldots , m_1') \in M$. If we let $\theta$ denote the outer conjugation for $\GL_N(E)$ with $N$ varies by abuse of notation, then $\theta(m)= ( \theta(m_1'),\ldots, \theta(m_s') , \theta(m_0), \theta(m_s),\ldots, \theta(m_1) ).$ Let $\pi$ be an irreducible representation of $\GL_N(E)$ with underlying space $V_{\pi}$. Define ${}^{\theta}\pi$ the representation on the same space with $\GL_N(E)$-action
\[ {}^\theta \pi(g)v:= \pi(\theta (g)) v.\]
Then ${}^\theta \pi$ is isomorphic to the contragredient of $\pi$.

Let $\widehat{I}(\GL_N^{\theta}(E))$ denote the space of twisted invariant distribution on $\GL_N(E)$, i.e., distributions invariant under the $\theta$-conjugation 
\[ \Ad_{\theta}(g)x:= gx\theta(g)^{-1}.  \]
Suppose $\pi^+$ is a representation of $\GL_N^+(E)$. Consider the distribution $\Theta(\pi^+)_{\theta}$, called the twisted character of $\pi^+$, that sends $f \in \mathcal{C}_{c}^{\infty}(\GL_N(E))$ to the trace of the finite rank operator 
\[ \pi^+(f)_{\theta}:= \int_{ \GL_N(E)} f(g ) \pi^+(g \theta) dg.  \]
Then $\Theta(\pi^+)_{\theta} \in \widehat{I}(\GL_N^\theta(E))$, and hence we obtain a map from $K\Pi(\GL_N^+(E))$, the Grothendieck group of finite length admissible representations  of $\GL_N^+(E)$, to $\widehat{I}(\GL_N^{\theta}(E))$. It is not hard to check that representations induced from $\GL_N(E)$ to $\GL_N^+(E)$ lie in and in fact span the kernel of this map. In other words, the following sequence is exact
 \begin{align}\label{eq exact sequence twisted character}
     K\Pi(\GL_N(E)) \xrightarrow[]{\Ind} K\Pi(\GL_N^+(E)) \xrightarrow[]{\Theta(\cdot)_{\theta}} \widehat{I}(\GL_N^\theta(E)). 
 \end{align}

Finally, let $\widetilde{\Tran}$ denote the twisted endoscopic transfer, which is a linear map
\[ \widetilde{\Tran}: \widehat{SI}(G) \to \widehat{I}(\GL_N^{\theta}(E)), \]
 where $\widehat{SI}(G)$ is the space of stable invariant distributions on the classical group $G$. See \cite[\S 2.1]{Art13} for more precise definition. See also \cite{KS99}. We shall only need the facts that it is well-defined, injective (\cite[Corollary 2.1.2]{Art13}) and compatible with Aubert-Zelevinsky involution (see \cite[\S A]{Xu17b} and \eqref{eq AZ Xu} below). For any local Arthur parameter $\psi$ of $G$, the stable distribution $\eta_{\psi}$ in Conjecture \ref{conj 2.2.1}(a) is characterized by
 \[ \widetilde{\Tran} (\eta_{\psi})= \Theta( \pi_{\psi}^+)_{\theta},\]
where $\pi_{\psi}$ is the irreducible representation in the local Arthur packet (which is an $L$-packet) $\Pi_{\psi_{\GL}}$ of $\GL_N(E)$, and $\pi_{\psi}^+$ is an irreducible representation of $\GL_N^+(E)$ extended from $\pi_{\psi}$ via Whittaker normalization, which we explain in the next subsection.

\subsection{Whittaker normalization and M{\oe}glin-Waldspurger normalization}
 Suppose that ${}^\theta \pi \cong \pi$, i.e. $\pi$ is self-dual. There are exactly two normalizations of isomorphism $T:\pi \xrightarrow[]{\sim} {}^\theta \pi$ such that $T^2=\text{id}$. If $T$ is any such normalization, then $-T$ is the other. Any choice of $T$ extends $\pi$ as a representation of $\GL_N^+(E)$ by $\pi(\theta):= T$. In this subsection, we recall two systematic ways to specify a choice of $T$, the Whittaker normalization (\cite[\S 2.2]{Art13}), and the M{\oe}glin-Waldspurger normalization \cite[\S 1.12]{MW06}.

First, we recall the notation of parabolic induction and Jacquet module for $\GL_N^+(E)$. Let $M$ be a $\theta$-invariant Levi subgroup and let $M^+:= M \rtimes \langle \theta \rangle$. Suppose $\pi$ (resp. $\sigma$) is a representation of $\GL_N^+(E)$ (resp. $M^+$), which may be regarded as a representation of $\GL_N(E)$ (resp. $M$) equipped with an action of $\theta$. Then as a representation of $M$ (resp. $\GL_N(E)$), $\Jac_{M}(\pi)$ (resp. $\Ind_M^{\GL_N(E)}$) naturally carries an action of $\theta$ also (see \cite[\S 3]{Rog88}). By abuse of notation, we still denote these maps by
\[ \Ind_{M}^{\GL_N(E)}: K\Pi(M^+) \to K\Pi(\GL_N^+(E)),\ \Jac_{M}: K\Pi(\GL_N^+(E)) \to K\Pi(M^+). \]


Now we recall the definition of Whittaker normalization from \cite[\S 2.2]{Art13}. Recall that we have fixed a $\theta$-stable Whittaker datum from the beginning. Let $\pi$ be an irreducible representation of $\GL_N(E)$ fixed by $\theta$. First, suppose $\pi$ is generic and take a Whittaker functional $\omega: V_{\pi} \to \BC$. Then we take $\theta_{W}(\pi)$ to be the unique isomorphism from $ \pi \to {}^\theta \pi$ that fixes the Whittaker functional, i.e.,
\[ \omega \circ \theta_W(\pi) =\omega.\]
Next, if $\pi$ is not generic, then we realize $\pi$ as a unique irreducible subrepresentation of its standard module $\Ind_M^{\GL_N(E)} (\sigma)$ via the (subrepresentation version of the) Langlands classification. Note that $M$ is fixed by $\theta$ and ${}^{\theta} \sigma \cong \sigma$ by the uniqueness of standard module since ${}^\theta\pi \cong \pi$. Also, $\sigma$ is essentially tempered, which implies that it is generic and hence $\theta_W(\sigma)$ is already defined. Therefore, $\Ind_M^{\GL_N(E)}(\sigma)$ carries an action of $\theta$ induced from $\theta_{W}(\sigma)$. Then $\theta_{W}(\pi)$ is defined as the restriction of this action to the subrepresentation $\pi$. In the rest of this section, we shall write $\pi^+$ (resp. $\pi^-$) the irreducible representation of $\GL_N^+(E)$ with $\theta$ acts by $\theta_W(\pi)$ (resp. $-\theta_W(\pi)$).


In \cite{MW06}, M{\oe}glin-Waldspurger introduced another normalization for $\pi_{\psi}$, which is easier to keep track of in the induction process. We recall this normalization for discrete or anti-discrete $\psi$ for simplicity. Write
\[\psi= \bigoplus_{\rho} \bigoplus_{i \in I_{\rho}} \rho \otimes S_{a_i}\otimes S_{b_i}.\]
By our assumption that $\psi$ or $\widehat{\psi}$ is discrete, we have $\rho \cong {}^{\theta} \rho$, and $a_i+b_i$ are of the same parity for any $i \in I_{\rho}$ fixing $\rho$. If $\psi$ is discrete, i.e. $b_i=1$ for any $i \in \sqcup_{\rho} I_{\rho}$, then 
\[ \pi_{\psi}= \bigtimes_{\rho} \bigtimes_{i \in I_{\rho}} \St(\rho,a_i),\]
where $\St(\rho,a_i)$ is the unique irreducible subrepresentation of 
\begin{align*}
    \rho|\cdot|^{\half{a_i-1}} \times     \rho|\cdot|^{\half{a_i-3}} \times \cdots \times     \rho|\cdot|^{\half{1-a_i}}.
\end{align*}
If $\psi$ is anti-discrete, i.e. $a_i=1$ for any $i \in \sqcup_{\rho} I_{\rho}$, then 
\[ \pi_{\psi}= \bigtimes_{\rho} \bigtimes_{i \in I_{\rho}} \Speh(\rho,b_i),\]
where $\Speh(\rho,b_i)$ is the unique irreducible subrepresentation of 
\[  \rho|\cdot|^{\half{1-b_i}} \times     \rho|\cdot|^{\half{3-b_i}} \times \cdots \times     \rho|\cdot|^{\half{b_i-1}}.\]
For $i \in I := \sqcup_{\rho} I_{\rho}$, let $B_i:= |\half{a_i-b_i}|$ and $\zeta_i:=1$ if $a_i\geq b_i$ and $\zeta_i:=-1$ otherwise. 

Now we define the normalization $\theta_{MW}(\pi_{\psi})$ inductively. Suppose $B_i=0$ for any $i\in I$. Then $\pi_{\psi}$ is a product of supercuspidal representations, and hence $\pi_{\psi}$ is generic. Define $\theta_{MW}(\pi_{\psi}):= \theta_W(\pi_{\psi})$ in this case. We shall include $\psi=0$ and $N=0$ in this case.

Suppose $B_i \neq 0$ for some $i \in I$. Take $\rho$ such that one of the following holds.
\begin{enumerate}
    \item [(1)] $ \min_{i \in I_{\rho}} \{B_i \} >0$.
    \item [(2)] $ \min_{i \in I_{\rho}} \{B_i\}=0$ but $|I_{\rho}|\geq 2$.
\end{enumerate}
In Case (1), take $j \in I_{\rho}$ such that $B_j=\min_{i \in I_{\rho}} \{B_i \}$. We have an injection
\begin{align}\label{eq MW normalization (1)}
    \pi_{\psi} \hookrightarrow \rho|\cdot|^{\zeta_i B_i} \times \pi_{\psi^-} \times \rho|\cdot|^{-\zeta_i B_i},
\end{align}
where $\psi^-:= \psi- \rho \otimes S_{a_j}\otimes S_{b_j}+ \rho \otimes S_{|a_j-2|}\otimes S_{|b_j-2|} $. Recall that $\min(a_j,b_j)=1$ by our assumption and if $a_j-2=0$ or $b_j-2=0$, ignore that term. Then $\theta_{MW}(\pi_{\psi})$ is defined to be the restriction of the action on $\rho|\cdot|^{\zeta_i B_i} \times \pi_{\psi^-} \times \rho|\cdot|^{-\zeta_i B_i} $ induced from $A^{-1}\otimes\theta_{MW}(\pi_{\psi^-}) \otimes A$, where $A$ is any choice of isomorphism from $\rho|\cdot|^{\zeta_i B_i}$ to $  {}^\theta (\rho|\cdot|^{\zeta_i B_i})\cong \rho|\cdot|^{-\zeta_i B_i}$.

In Case (2), take $j_1, j_2 \in I_{\rho}$ such that $B_{j_1}=0$ and $B_{j_2}=\min_{i \in I_{\rho}\setminus \{j_1\}} \{B_i \}$. Then let $\tau$ (resp. ${}^\theta \tau$) be the unique irreducible subrepresentation of 
\[ \rho|\cdot|^{\zeta_{j_2} B_{j_2}} \times \cdots \times  \rho|\cdot|^{0}\ \  \text{(resp.} \rho|\cdot|^{0} \times \cdots \times  \rho|\cdot|^{-\zeta_{j_2} B_{j_2}} \text{)}.\]
We have an injection
\begin{align}\label{eq MW normalization (2)}
   \pi_{\psi} \hookrightarrow \tau \times \pi_{\psi^-} \times {}^{\theta}\tau,
\end{align}
where $\psi^-:= \psi- \rho \otimes S_{a_{j_1}}\otimes S_{b_{j_1}} - \rho \otimes S_{a_{j_2}}\otimes S_{b_{j_2}}$. Then $\theta_{MW}(\pi_{\psi})$ is defined to be the action induced from $\theta_{MW}(\pi_{\psi^{-}})$ similarly as Case (1).

It turns out that if $\psi$ is discrete or anti-discrete, then $\theta_W(\pi_{\psi})=\theta_{MW}(\pi_{\psi})$. This is a special case of \cite[Theorem 5.6.1]{MW06}. {The same argument applies to the case of unitary group and we do not repeat the details here.}

\begin{thm}[{\cite[Theorem 5.6.1]{MW06}}]\label{thm MW06 sec 5}
    Suppose $\psi$ is discrete or anti-discrete. Then $\theta_W(\pi_{\psi})=\theta_{MW}(\pi_{\psi})$.
\end{thm}

\subsection{\texorpdfstring{Computation of Aubert-Zelevinsky involution}{}}
In this subsection, we recall the definition of $ \textrm{inv}^{\theta}: K\Pi(\GL_N^+(E)) \to K\Pi(\GL_N^+(E)) $ defined in \cite[\S 6.3]{Xu17b}, and then compute $\Theta( \Inv^{\theta}(\pi_{\phi}^+))_{\theta}$ for discrete local Arthur parameter $\phi$. The argument is almost the same as \cite[Lemma 3.2.1]{MW06}. There are slight differences because they consider the generalized Aubert-Zelevinsky involution, while we consider the full Aubert-Zelevinsky involution. The proof there assumed $a_i+b_i$ are all odd for simplicity, and we include the other case for completeness.

We give some motivation before stating the definition of $\textrm{inv}^{\theta}$. First, we recall the long exact sequence in \cite[Theorem 3.6]{Aub95}. Let $\pi$ be an irreducible representation of $\GL_N(E)$ such that ${}^\theta \pi \cong \pi$. Write $V$ for the underlying space of $\pi$. Recall that we let $\Delta=\{\alpha_1,\ldots, \alpha_{N-1}\}$ be the set of (restricted) simple roots and $\theta(\alpha_i)=\alpha_{N-i}$. For any $I \subseteq \Delta$, let $V_I$ denote the underlying space of $\Ind_{M_I}^{\GL_N(E)} \circ \Jac_{M_I}(\pi)$. Then for any $\Delta \supseteq J \supseteq I$, there is a natural map $\phi_{I}^J: V_{J} \to V_{I}$ such that 
\[ \phi^{K}_I= \phi_{I}^J \circ \phi^{K}_J \]
for any $\Delta \supseteq K \supseteq J \supseteq I$. Suppose $\Delta\setminus I= \{ \alpha_{n_1},\ldots, \alpha_{n_s}  \}$ with $n_1 <\cdots< n_s$. Set
\[ e^I:= e^{\alpha_{n_1}} \wedge \cdots \wedge e^{\alpha_{n_s}} \in \Lambda^{|\Delta\setminus I|} (\BC^{|\Delta\setminus I|}).\]
Then for any $J= I \sqcup \{\alpha_m\}$, define $\xi_{I}^J \in \{\pm 1\}$ such that
\[ e^I= \xi^J_I \cdot e^J  \wedge e^{\alpha_m}.\]
Now consider the following sequence
\[ 0 \to V \xrightarrow[]{d_{|\Delta|}} \bigoplus_{|J|=|\Delta|-1} V_J \xrightarrow[]{d_{|\Delta|-1}} \bigoplus_{|J|=|\Delta|-2} V_J \xrightarrow[]{d_{|\Delta|-2}} \cdots \xrightarrow[]{d_{1}} V_{\emptyset} \to 0,   \]
where
\[ d_{i}:=  \bigoplus_{J \supset I,\ |J|=|I|+1=i+1} \xi^{J}_I \cdot \phi^J_I.\]
It follows from the definition that
\begin{enumerate}
    \item [$\oldbullet$] we have $d_{i} \circ d_{i+1}=0$, i.e. it is a complex;
    \item [$\oldbullet$] suppose $M_{I_0}$ is a Levi subgroup corresponding to the cuspidal support of $\pi$. Then $V_{J}=0$ for any $|J|<|I_0|$.
\end{enumerate}
A crucial step to show that $D_{\GL_N(E)}$ sends irreducible representations to irreducible representations modulo sign is to verify that the complex is indeed exact.

\begin{thm}[{\cite[Theorem 3.6, corollary 3.9]{Aub95}}]\label{thm Aubert}
   Suppose $M_{I_0}$ is a Levi subgroup corresponding to the cuspidal support of $\pi$. Then the following sequence is exact.
\begin{align}\label{eq seq aubert}
    0 \to V \xrightarrow[]{d_{|\Delta|}} \bigoplus_{|J|=|\Delta|-1} V_J \xrightarrow[]{d_{|\Delta|-1}} \bigoplus_{|J|=|\Delta|-2} V_J \xrightarrow[]{d_{|\Delta|-2}} \cdots \xrightarrow[]{d_{|I_0|+1}} \bigoplus_{|J|=|I_0|} V_{J} .
\end{align}    
In particular, 
\[(-1)^{|I_0|} \textrm{coker} (d_{|I_0|+1})= \sum_{J \subseteq \Delta} (-1)^{|J|} \Ind_{M_J}^{\GL_N(E)} \circ \Jac_{M_J}(\pi).\]
If $\pi$ is irreducible, then so is $ \widehat{\pi}:=\textrm{coker} (d_{|I_0|+1})$.
\end{thm}

Note that if ${}^\theta \pi \cong \pi,$ then $ {}^\theta \widehat{\pi} \cong \widehat{\pi}$ since Aubert-Zelevinsky involution commutes with contragredient (\cite[Theorem 1.7(1)]{Aub95}).

Now we introduce the action of $\theta$. Let $\{\pi^+,\pi^-\}$ be two extensions of $\pi$ where $\theta$ act on $\pi^+$ via the Whittaker normalization $\theta_W(\pi)$. Each term in the exact sequence \eqref{eq seq aubert} is already equipped with a $\theta$ action induced from $\theta_{W}(\pi)$. That is, we have
\[ \theta_J: \Ind_{M_J}^{\GL_N(E)} \circ \Jac_{M_J} (\pi^+) \to \Ind_{M_{\theta(J)}}^{\GL_N(E)} \circ \Jac_{M_{\theta(J)}} (\pi^+).\]
However, it is not always true that $ \bigoplus_{|J|=j}\theta_J$ commutes with $d_{j}$ because of the signs $\xi^J_I$. Indeed, from $\theta(\alpha_{i})=\alpha_{N-i}$, one can check that for $J=I\sqcup \{\alpha_m\}$, 
\[ \xi^J_I \cdot \xi^{\theta(J)}_{\theta(I)}= (-1)^{|\Delta|-| I|}= (-1)^{\left\lfloor\half{|\Delta|-|I|} \right\rfloor} (-1)^{\left\lfloor\half{|\Delta|-|J|} \right\rfloor}. \]
Therefore, if we let $\theta$ acts on the exact sequence \eqref{eq seq aubert} via $\theta_j:= (-1)^{\left\lfloor\half{|\Delta|-j} \right\rfloor}\bigoplus_{|J|=j} \theta_J$, then it becomes an exact sequence of  representations of $\GL_N^+(E)$. As a consequence, let $\varepsilon_j= (-1)^{\left\lfloor\half{|\Delta|-j} \right\rfloor}$. There exists an $\varepsilon' \in \{+,-\}$ such that
\[ (-1)^{|I_0|}(\widehat{\pi})^{\varepsilon'}= \sum_{J \subseteq \Delta} (-1)^{|J|} \Ind_{M_J}^{\GL_N(E)} \circ \Jac_{M_J}(\pi^{\varepsilon_{|J|}} ).\]
We may simplify the notation by passing to the twisted characters. Observe that if $\theta(J)=J$, then 
\[ \Theta(\Ind_{M_J}^{\GL_N(E)} \circ \Jac_{M_J}(\pi^{\varepsilon_{|J|}} ))_{\theta}= \varepsilon_{|J|}  \Theta(\Ind_{M_J}^{\GL_N(E)} \circ \Jac_{M_J}(\pi^+ ))_{\theta}.\]
On the other hand, if $\theta(J) \neq J$, then $\theta_{|J|}$ exchanges $V_J$ and $V_{\theta(J)}$, which implies that
\[ \Theta(\Ind_{M_J}^{\GL_N(E)} \circ \Jac_{M_J}(\pi^{\varepsilon_{|J|}} ) \oplus  \Ind_{M_{\theta(J)}}^{\GL_N(E)} \circ \Jac_{M_{\theta(J)}}(\pi^{\varepsilon_{|\theta(J)|}} ) )_{\theta}=0.\]
Therefore, we obtain that
\begin{align}\label{eq Theta AZ}
     (-1)^{|I_0|}\varepsilon' \Theta((\widehat{\pi})^+)_{\theta}= \Theta\left(  \sum_{J \subseteq \Delta, \theta(J)=J} (-1)^{\left\lfloor\half{|\Delta|-|J|} \right\rfloor+|J|} \Ind_{M_J}^{\GL_N(E)} \circ \Jac_{M_J}(\pi^{+} )\right)_{\theta}.
\end{align}
Finally, let $(A_{M_J})_{\theta}$ be the $\theta$-coinvariant of the maximal split central torus of $M_J$. Then $ \dim((A_{M_J})_{\theta})= \left\lfloor \half{|\Delta|-|J|+1} \right\rfloor$, and hence 
\[ (-1)^{\left\lfloor\half{|\Delta|-|J|} \right\rfloor+|J|}=(-1)^{\dim((A_{M_J})_{\theta})} (-1)^{|\Delta|}.\]
This leads to the definition of $\textrm{inv}^{\theta}$ in \cite[\S A]{Xu17b}.

\begin{defn}\label{def twisted Aubert dual}
Suppose $\pi$ is a representation of $\GL_N(E)$ fixed by $\theta$. Let $\pi^+$ be an extension of $\pi$. Define 
\[ \Inv^{\theta}(\pi^+):= \sum_{M \in \mathcal{M}^{\theta} } (-1)^{\dim(( A_{M})_{\theta})} \Ind_{M}^{\GL_N(E)} \circ \Jac_{M}(\pi^{+} ), \]
which is an element in $K\Pi(\GL_N^+(E))$. Here $\mathcal{M}^{\theta}$ is the set of $\theta$-invariant standard Levi subgroups, and $(A_M)_{\theta}$ is the $\theta$-coinvariant of the maximal split central torus of $M$.
\end{defn}

Now we prove the following lemma based on the argument in \cite[Lemma 3.2.2]{MW06}.

\begin{lemma}\label{lem inv theta}
    Let $\phi$ be a discrete local Arthur parameter of a classical group $G$ and let $\psi=\widehat{\phi}$. Then 
    \[\Theta( \Inv^{\theta}(\pi_{\phi}^+))_{\theta}=(-1)^{r(G)}\beta(\phi_{\psi}) \Theta(\pi_{\psi}^+)_{\theta}.\]
\end{lemma}
\begin{proof}
    From \eqref{eq Theta AZ}, we already have $\Theta(\textrm{inv}^{\theta}(\pi_{\phi}^+))_{\theta}= \varepsilon \Theta(\pi_{\psi}^{+})_{\theta}$ for some $\varepsilon \in \{\pm 1\}$. We compute this sign explicitly. 
    First, we construct certain Levi subgroup $M_{I_{\psi}}$ and a representation $\sigma$ on it. Write
    \[ \phi= \bigoplus_{\rho \in R} \bigoplus_{i \in I_{\rho}} \rho \otimes S_{a_i} \otimes S_1.\]
    Decompose $R=R_{odd} \sqcup R_{even}$, where $R_{odd}=\{\rho \in R \ | \ a_i+ b_i \text{ are odd }\forall i \in I_{\rho}\}$. Write $I_{\rho}=\{1,\ldots, |I_{\rho}|\}$ where $a_1<\cdots < a_{|I_{\rho}|}$. For $\rho \in R_{odd}$, define
    \[ \tau_{\rho,i}:= \rho|\cdot|^{\half{1-a_i}} \otimes \cdots \otimes \rho|\cdot|^{\half{-1}},\  {}^\theta\tau_{\rho,i}:= \rho|\cdot|^{\half{1}} \otimes \cdots \otimes \rho|\cdot|^{\half{a_i-1}}\]
for each $i \in I_{\rho}$. For $\rho \in R_{even}$, define
\[\tau_{\rho,i}:= \rho|\cdot|^{\half{1-a_{i}}} \otimes \cdots \otimes \rho|\cdot|^{1},\ {}^\theta \tau_{\rho,i}:=  \rho|\cdot|^{1} \otimes \cdots \otimes \rho|\cdot|^{\half{a_{i}-1}}\]
if $i \in I_{\rho}=\{1,\ldots, |I_{\rho}|\}$ is odd, and
\[\tau_{\rho,i}:= \rho|\cdot|^{\half{1-a_{i}}} \otimes \cdots \otimes \rho|\cdot|^{0},\ {}^\theta \tau_{\rho,i}:=  \rho|\cdot|^{0} \otimes \cdots \otimes \rho|\cdot|^{\half{a_{i}-1}}\]
if $i \in I_{\rho}$ is even. Then for any $\rho \in R$ define  
\[ \tau_{\rho}:= \tau_{\rho,1} \otimes \cdots \otimes \tau_{\rho,|I_{\rho}|},\   {}^\theta\tau_{\rho}:= {}^\theta\tau_{\rho,|I_{\rho}|} \otimes \cdots \otimes {}^{\theta}\tau_{\rho,1}. \]
Finally, define 
\[ \sigma:= \bigotimes_{\rho \in R} \tau_{\rho} \otimes \left(\bigtimes_{\rho \in R_{even}, |I_{\rho}|\text{ is odd}}\rho \right)  \otimes \bigotimes_{\rho \in R} {}^\theta \tau_{\rho},   \]
where we fix any order for the first product and the opposite order for the last product so that ${}^\theta \sigma \cong \sigma$. Let $M_{I_{\psi}}$ be the corresponding Levi subgroup. It is not hard to see from \eqref{eq beta(phi_psi)} that
\begin{align}\label{eq phi_psi, A_M}
    (-1)^{r(G)} \beta(\phi_{\psi})=\dim((A_{M_{I_{\psi}}})_{\theta}).
\end{align}
On the other hand, let $\Ind_{M_{I_{\psi}}}^{M_{I_{\psi}}\rtimes \langle \theta  \rangle }(\sigma)= \sigma^+ \oplus \sigma^-$, where $\theta$ acts on $\sigma^+$ via the Whittaker normalization $\theta_{W}(\sigma)$, which is the same as the M{\oe}glin-Waldspurger normalization $\theta_{MW}(\sigma)$ by Theorem \ref{thm MW06 sec 5}. Then by the definition of $\theta_{MW}(\sigma)$, we see that for $\varepsilon' \in \{\pm 1\}$, the semisimplification of $\Jac_{M_{I_{\psi}}}(\pi_{\psi}^{\varepsilon'})$ contains $\sigma^{\varepsilon'}$ of multiplicity one but does not contain $\sigma^{-\varepsilon'}$. Thus for any representation $\tau$ of $M_{I_{\psi}} \rtimes \langle \theta \rangle$, let $\widetilde{m}(\tau)$ denote the multiplicity of $\sigma^{+}$ in the semisimplification of $\tau$ minus that of $\sigma^-$. Our goal becomes the computation of $ \widetilde{m}(\Jac_{M_{I_{\psi}}}\circ\Inv^{\theta}(\pi_{\phi}^+))$, which is equal to the following expression
\begin{align}\label{eq goal twisted AZ}
   \sum_{J \subseteq \Delta, \theta(J)=J } (-1)^{\dim(( A_{M_J})_{\theta})} \widetilde{m}( \Jac_{M_{I_{\psi}}} \circ \Ind_{M_J}^{\GL_N(E)} \circ \Jac_{M_J}(\pi_{\phi}^{+} )) .
\end{align} 

Let $W$ denote the Weyl group of $\GL_N(E)$, $\Phi^+$ denote the set of positive roots and let 
\[\mathcal{D}(I,J):= \{w\in W\ | \ w^{-1}(I) \subseteq \Phi^+,\ w(J) \subseteq \Phi^+  \}\]
for any $I,J \subseteq \Delta$. The geometric lemma (\cite[\S 2.11]{BZ77}) implies that
\[  \Jac_{M_{I_{\psi}}} \circ \Ind_{M_J}^{\GL_N(E)} \circ \Jac_{M_J}(\pi_{\phi})=  \sum_{w \in \mathcal{D}(I_{\psi},J)} \Ind_{M_{I_{\psi}}\cap wM_{J}w^{-1}}^{M_{I_{\psi}}} \Ad(w) \left( \Jac_{M_{J}\cap w^{-1}M_{I_{\phi}}w}\pi_{\psi}^+\right).  \]

For $J \subset \Delta$ such that $\theta(J)=J$ and $w \in \mathcal{D}(I_{\psi},J)$, let $\Pi(w,J)$ denote each term in the right hand side above. If $\theta(w) \neq w$, then $\theta$ exchanges $\Pi(w,J)$ with $\Pi(\theta(w),J)$ and hence $\widetilde{m}( \Pi(w,J)+\Pi(\theta(w),J))=0$. Thus we may ignore these terms. We may further ignore those $\Pi(w,J)$ such that $M_{I_{\psi}} \not\subseteq wM_{J}w^{-1}$. Indeed, if $K$ is a $\theta$-invariant proper subset of $I_{\psi}$ and $\tau$ is an irreducible representation of $M_{K}$ such that $\sigma$ is a subquotient of $\Ind_{M_K}^{M_{I_{\psi}}} \tau$, then ${}^\theta \tau \not\cong\tau$ by the form of $\sigma$. Therefore, if $M_{I_{\psi}} \not\subseteq wM_{J}w^{-1}$, there is a decomposition $\Pi(w,J)= V_1 \oplus V_2$ such that $\theta$ exchanges $V_1, V_2$, and hence $\widetilde{m}(\Pi(w,J))=0$. Note that if $M_{I_{\psi}} \subseteq w M_{J}w^{-1}$ and $w \in \mathcal{D}(I,J)$, then $w^{-1}(I_{\psi}) \subseteq J \subseteq \Delta$.

In summary, we may simplify \eqref{eq goal twisted AZ} as follows.
\begin{align}\label{eq goal twisted AZ 2}
  \sum_{J \subseteq \Delta, \theta(J)=J } (-1)^{\left\lfloor \half{|\Delta|-|J|+1} \right\rfloor} \sum_{\substack{w \in \mathcal{D}(I_{\psi},J),\ \theta(w)=w, \\  w^{-1}(I_{\psi}) \subseteq J}} \widetilde{m}\left(\Ad(w)   \left( \Jac_{ w^{-1}M_{I_{\psi}}w}\pi_{\phi}^+\right)\right) .
\end{align} 

Next, we would like to exchange the order of the summation. Fixing a $w$ such that $w^{-1}(I_{\psi}) \subseteq \Delta$.  The subsets $J$ that appear in the summation related to $w$ are exactly the subsets of $\Delta$ satisfying 
\begin{enumerate}
    \item [$\oldbullet$] $\theta(J)=J$,
    \item [$\oldbullet$] $w (J) \subseteq \Phi^+$,
    \item [$\oldbullet$] $w^{-1}(I_{\psi}) \subseteq J .$
\end{enumerate}
 Thus, let $\Delta^{w}:=\{\alpha \in \Delta\ | \ w(\alpha)\in \Phi^+\} \setminus w^{-1}(I_{\psi})$ and let $\Delta^w_{\theta}$ be the set of $\theta$-cosets of $\Delta^w$. Any $J$ that appears in the summation related to $w$ can be written uniquely as $J_{\theta}^w \sqcup w^{-1}(I_{\psi})$ for some $J_{\theta}^w \subseteq \Delta_{\theta}^w$. Therefore, we may rewrite \eqref{eq goal twisted AZ 2} as
 \begin{align}\label{eq goal twisted AZ 3}
    \sum_{\substack{w \in W, \theta(w)=w,\\ w^{-1}(I_{\psi})\subseteq \Delta}} \left(
    \sum_{J^w_{\theta} \subseteq \Delta^{w}_{\theta}} (-1)^{\half{|J_{\theta}^w|}} \right) \cdot (-1)^{\left\lfloor \half{|\Delta|-|I_{\psi}|+1} \right\rfloor} \cdot \widetilde{m}\left(\Ad(w)   \left( \Jac_{ w^{-1}M_{I_{\psi}}w}\pi_{\phi}^+\right)\right).
\end{align} 
The alternating sum is non-zero only when $\Delta_{\theta}^w$ is empty. That is, $w=w_0$, the unique element in $W$ such that $w_0^{-1}(I_{\psi}) \subseteq \Phi^+$ and $w_0^{-1}(\Delta\setminus I_{\psi}) \subseteq -\Phi^+$. Thus $w_0^{-1}M_{I_{\psi}} w_0= M_{I_{\psi}}$ and 
\begin{align*}
    \Ad(w_0)^{-1}\sigma= \bigotimes_{\rho \in R} \widetilde{\tau_{\rho}} \otimes \left(\bigtimes_{\rho \in R_{even}, |I_{\rho}|\text{ is odd}}\rho \right)  \otimes \bigotimes_{\rho \in R} \widetilde{{}^\theta \tau_{\rho}},
\end{align*}
where the order of the product is the same as the order taken in the definition of $\sigma$ and $\widetilde{(\cdot)}$ is the contragredient. By the definition of M{\oe}glin-Waldspurger normalization, $\Jac_{M_{I_{\psi}}} \pi_{\phi}^+$ contains $\Ad(w_0)^{-1}(\sigma^+)$ of multiplicity one but does not contain $\Ad(w_0)^{-1}(\sigma^-)$. We conclude that 
\[ \widetilde{m}(\Jac_{M_{I_{\psi}}}\circ\Inv^{\theta}(\pi_{\phi}^+))= (-1)^{\left\lfloor \half{|\Delta|-|I_{\psi}|+1} \right\rfloor},\]
and hence by \eqref{eq phi_psi, A_M}, we have
\[\Theta( \Inv^{\theta}(\pi_{\phi}^+))_{\theta}=(-1)^{r(G)}\beta(\phi_{\psi}) \Theta(\pi_{\psi}^+)_{\theta}.\]
This completes the proof of the lemma.
 \end{proof}

\subsection{Proof of Proposition \ref{prop twisted}}
The last result we need is the compatibility of Aubert-Zelevinsky involution and twisted endoscopic transfer, which we recall now.
Though $\Inv^{\theta}$ is defined on $K\Pi(\GL_N^+(E))$ in Definition \ref{def twisted Aubert dual}, according to the exact sequence \eqref{eq exact sequence twisted character}, we may regard it as an operator on $\widehat{I}(\GL_N^{\theta}(E))$ by abuse of notation. That is, if $\eta= \Theta(\pi^+)_{\theta}$ for certain $\pi^+ \in K\Pi(\GL_N^+(E))$, then
\begin{align*}
     \Inv^{\theta} ( \eta ):= \Theta( \Inv^{\theta}(\pi^+) )_{\theta}
\end{align*}
is independent of the choice of $\pi^+$. Let $\Inv_{G}:= (-1)^{r(G)} D_{G}$ where $D_G$ is the Aubert-Zelevinsky involution on the classical group $G$ defined in \eqref{eq AZ}. Then the following equality is proved in \cite[\S A]{Xu17b} 
\begin{align}\label{eq AZ Xu}
\widetilde{\Tran} \circ \Inv_{G} = \Inv^{\theta} \circ \widetilde{\Tran}.    
\end{align}
Now we prove Proposition \ref{prop twisted}.

\begin{proof}[proof of Proposition \ref{prop twisted}]
 Let $\phi$ be a tempered local Arthur parameter of $G$ and let $\psi=\widehat{\phi}$. Let $\pi_{\phi}$ (resp. $\pi_{\psi}$) be the irreducible representation of $\GL_N(E)$ corresponding to the local Arthur parameter $\phi_{\GL}$ (resp. $\psi_{\GL}$) and $\pi_{\phi}^+$ (resp. $\pi_{\psi}^+$) the extension with respect to the Whittaker normalization. The stable distribution $\eta_{\phi}$ (resp. $\eta_{\psi}$) is characterized by 
\[ \widetilde{\Tran}(\eta_{\phi})= \Theta(\pi_{\phi}^+)_{\theta}\  (\text{resp. }\widetilde{\Tran}(\eta_{\psi})= \Theta(\pi_{\psi}^+)_{\theta}).\]

 First, suppose $\phi$ is discrete. Then Lemma \ref{lem inv theta} implies that
  \begin{align*}
 \widetilde{\Tran} ( D_{G}(\eta_{\phi}))&=  (-1)^{r(G)} \widetilde{\Tran} ( \Inv_{G} (\eta_{\phi}))\\
 &= (-1)^{r(G)}\textrm{inv}^{\theta} (\widetilde{\Tran}(\eta_{\phi}))\\
 &= (-1)^{r(G)}\textrm{inv}^{\theta} (\pi_{\phi}^+)\\
&=  \beta(\phi_{\psi}) \Theta( \pi_{\psi}^+)_{\theta}\\
 &= \beta(\phi_{\psi})\widetilde{\Tran} (\eta_{\psi}).
 \end{align*}
Since $\widetilde{\Tran}$ is an injection (\cite[Corollary 2.1.2]{Art13}), we have verified that $ D_G(\eta_{\phi})=\beta(\phi_{\psi}) \eta_{\psi} $ when $\phi$ is discrete.

For general tempered $\phi$, write $\phi= \phi_0+ (\phi_1+ {}^\sigma \phi_1^{\vee})$ where $\phi_0$ is discrete. Let $\psi=\widehat{\phi}$ and $\psi_i= \widehat{\phi_i}$ for $i =0,1$. By Proposition \ref{prop reduction to discrete}, we have
\[ \eta_{\phi}= \pi_{\phi_1} \rtimes \eta_{\phi_0} \ (\text{resp. } \eta_{\psi}= \pi_{\psi_1} \rtimes \eta_{\psi_0}),\]
where $\pi_{\phi_1}$ (resp. $\pi_{\psi_1}$) is the unique irreducible representation in the local Arthur packet $\Pi_{\phi_1}$ (resp. $\Pi_{\psi_1}$) of $\GL_{\dim(\psi_1)}(E)$. Then by Observations (i), (ii) and (iii) in \S \ref{sec proof of prop eq of sign Moeglin}, we obtain
\begin{align*}
    D_G( \eta_{\phi})&=    D_G( \pi_{\phi_1}\rtimes \eta_{\phi_0})\\
    &= \beta( \pi_{\psi_1} )  \pi_{\psi_1} \rtimes  D_{G_0}(\eta_{\phi_0})\\
    &= \beta_{\GL}( \phi_{\psi_1} )  \pi_{\psi_1} \rtimes  (\beta(\phi_{\psi_0}) \eta_{\psi_0})\\
    &= \beta(\phi_{\psi}) \eta_{\psi}. 
\end{align*}
This completes the proof of Proposition \ref{prop twisted}.
\end{proof}

\section{Aubert-Zelevinsky involution of generic representations}\label{dual of generic reps}
In this section, we compute the $L$-parameter of the Aubert-Zelevinsky involution of generic representations of quasi-split classical groups, following the idea of \cite{Jan18} under an assumption (Working Hypothesis \ref{assu AZ}) which will be removed in future work.

First, we give some notations. Let $\pi$ be a generic representation of $G(V_{an,\mathfrak{r}})$ with $L$-parameter $\phi_{\pi}$. By the classification of generic dual (see \cite[Theorem 4.23]{JL24} for example), we may realize $\pi$ as an irreducible parabolic induction
\begin{align}\label{eq gen classification}
    \pi= \bigtimes_{i=1}^f \tau_f \rtimes \pi_{temp},
\end{align}
where $\tau_i$'s are generic representations of some $\GL_{d_i}(E)$ and $\pi_{temp}$ is a generic representation that lies in a tempered local Arthur packet $\Pi_{\phi_{temp}}$ of $G(V_{an,r})$ with $r \leq \mathfrak{r}$. Let $\psi:= \widehat{\phi_{temp}}$. Denote
\[ \widehat{\phi_{\pi}}:= \bigoplus_{i=1}^f  (\phi_{\widehat{\tau_i}}+{}^\sigma\phi_{\widehat{\tau_i}}^{\vee})+ \phi_{\psi}. \]
Remark that the restriction of $\phi_{\widehat{\tau_i}}$'s to $\SL_2(\BC)$ are trivial by the classification of generic representations of $\GL_{d_i}(E)$, and hence the restriction of the whole $\widehat{\phi_{\pi}}$ to $\SL_{2}(\BC)$ is also trivial. Indeed, $\phi_{\pi}$ and $\widehat{\phi_{\pi}}$ correspond to the unique open orbit and the closed orbit (the zero orbit) in the associated Vogan variety respectively, and they are the Pyasetskii involution of each other (see \cite{CDFZ22}, \cite[\S 4.2, 6.4]{CFMMX22}). Here is the main result of this subsection.

\begin{prop}\label{prop generic}
    In the setting above, for any generic representation $\pi$, we have $\phi_{\widehat{\pi}}= \widehat{\phi_{\pi}}$. In particular, $\phi_{\widehat{\pi}}|_{\SL_2(\BC)}$ is trivial.  
\end{prop}

Here are two applications of Proposition \ref{prop generic}.

\begin{remark}\label{remark generic dual}\ 
    \begin{enumerate}
        \item Assuming the closure ordering conjecture $($see \cite[Conjecture 3.1]{Xu24}, \cite[Conjecture 1.2]{HLLZ22}$)$ for local Arthur packets, Proposition \ref{prop generic} implies the enhanced Shahidi conjecture for quasi-split classical groups. See \cite[\S 6]{HLLZ22} for more discussion.
        \item Proposition \ref{prop generic} verifies \cite[Conjecture 1.1]{HLLS24} on the upper bound of wavefront sets for generic representations.
    \end{enumerate}
\end{remark}

Let us explain the connection between Proposition \ref{prop generic} and other results in the previous sections. As we shall see in the proof, it is straightforward to reduce the problem from $\pi$ to $\pi_{temp}$, so we assume $\pi$ is tempered generic, and hence $\phi_{\pi}=\phi_{temp}$. The representation $\pi=\pi(\phi, \varepsilon)$ is generic with respect to the fixed Whittaker datum if and only if $\varepsilon$ is trivial (see \cite[Proposition 8.3.2(a)]{Art13} and \cite{Ato17}). Thus, Theorem \ref{thm main} implies that 
\[ \widehat{\pi}=\pi(\psi, \varepsilon_{\psi}^{M/MW}).\]
Since $\varepsilon_{\psi}^{M/MW} \in \widehat{\mathcal{S}}_{\phi_{\psi}} \subseteq \widehat{\mathcal{S}}_{\psi}$ (see Remark \ref{rmk in Lpacket}), for quasi-split special orthogonal, symplectic, and unitary groups, it is a consequence of \cite[Proposition 7.4.1]{Art13} and \cite[Proposition 8.4.1]{Mok15} that $\widehat{\pi} \in \Pi_{\phi_{\psi}}.$ We shall give a different proof including unitary groups.
In the proof, we need an analogue of Lemma \ref{lem Jacquet module}(b-1), (b-2) and (c) for tempered $\phi$. We state it in the Working hypothesis below. This can be verified by similar argument in \cite[\S 5]{Ato20}, where he verified these statements for $\Sp_{2n}(F)$ and split $\SO_{2n+1}(F)$. A detailed proof for other groups will be provided in a future work.

\begin{assumption}\label{assu AZ}
    Suppose $\pi(\phi,\varepsilon)$ is a tempered representation of $G(V)$ of good parity. If $\phi$ contains $\rho \otimes S_{a}\otimes S_1$ of multiplicity $m$, and one of the following holds.
    \begin{enumerate}
        \item [(1)] $a>2$ and $\phi$ does not contain $\rho \otimes S_{a-2}\otimes S_1$.
        \item [(2)] $a>2$ and $\phi$ contains $\rho \otimes S_{a-2}\otimes S_1$ with $\varepsilon(\rho\otimes S_{a}\otimes S_1)\varepsilon(\rho\otimes S_{a-2}\otimes S_1)=1$.
        \item [(3)] $a=2$ and $\varepsilon(\rho\otimes S_{a}\otimes S_1)=1.$
    \end{enumerate}
    Then, let $\phi^{-}:= \phi- (\rho\otimes S_{a}\otimes S_1)^{\oplus m}+(\rho\otimes S_{a-2}\otimes S_1)^{\oplus m}$ and $\varepsilon^{-} \in \widehat{\mathcal{S}}_{\phi^-,\chi_{V^-}}$ be given by the same recipe in Lemma \ref{lem Jacquet module}. We have an injection
    \[\pi(\phi,\varepsilon) \hookrightarrow \underbrace{\rho|\cdot|^{\half{a-1}}\times\cdots \times \rho|\cdot|^{\half{a-1}}}_{m \text{ copies}}   \rtimes \pi(\phi^{-}, \varepsilon^{-}).   \]    
\end{assumption}

Now we prove Proposition \ref{prop generic}.
\begin{proof}[Proof of Proposition \ref{prop generic}.]

First, we reduce the problem from $\pi$ to $\pi_{temp}$. Taking Aubert-Zelevinsky involution on \eqref{eq gen classification}, we obtain
\[ \widehat{\pi}= \bigtimes_{i=1}^f \widehat{\tau_f} \rtimes \widehat{\pi_{temp}}. \]
This implies (for example, see \cite[Lemma 10.3]{HLLS24})
\[ \phi_{\widehat{\pi}}= \bigoplus_{i=1}^f  (\phi_{\widehat{\tau_i}}+{}^\sigma\phi_{\widehat{\tau_i}}^{\vee}) + \phi_{\widehat{\pi_{temp}}},\]
which completes the reduction. Furthermore, by Proposition \ref{prop reduction to discrete}, we may write
\[ \pi_{temp}= \tau_{bp} \rtimes \pi_{gp},\]
where $\pi_{gp}$ lies in a tempered local Arthur packet $\Pi_{\phi_{gp}}$ of good parity (also see \cite[Theorem 6]{Moe06a}). Then, the same argument reduces the problem to $\pi_{gp}$. Thus, we assume $\phi= \phi_{temp}$ is of good parity in the rest of the proof.

Write 
\begin{align*}
    \phi= \bigoplus_{\rho \in R_0} \bigoplus_{i=1}^{a_{\rho}} (\rho \otimes S_{2i} \otimes S_1)^{\oplus m_{\rho,i}}+ \bigoplus_{\rho \in R_1} \bigoplus_{i=0}^{a_{\rho}} (\rho \otimes S_{2i+1} \otimes S_1)^{\oplus m_{\rho,i}},
\end{align*}
where $ a_{\rho} \in \Z_{\geq 1} $ and $m_{\rho} \in \Z_{\geq 0}$. For $ 0 \leq i \leq a_{\rho}$, let $M_{\rho,i}:= \sum_{j=i}^{a_{\rho}} m_{\rho,i}$ and let
\begin{align*}
   \phi_0&:= \bigoplus_{\rho\in R_1} (\rho\otimes  S_1)^{\oplus M_{\rho,0}},\\
   \phi_1&:= \bigoplus_{\rho\in R_0} \bigoplus_{i=1}^{a_{\rho}} (\rho |\cdot|^{i-\half{1}} \otimes  S_1)^{\oplus M_{\rho,i}} + \bigoplus_{\rho\in R_1} \bigoplus_{i=1}^{a_{\rho}} (\rho |\cdot|^{ i} \otimes  S_1)^{\oplus M_{\rho,i}} .
\end{align*}
We have $\widehat{\phi_{\pi}}=\phi_{\psi}= \phi_0+(\phi_1+ {}^\sigma\phi_1^{\vee}) $. It suffices to show that
\begin{align}\label{eq goal AZ}
    \widehat{\pi} \hookrightarrow \bigtimes_{\rho \in R_0} \bigtimes_{i=1}^{a_{\rho}} (\rho|\cdot|^{-a_{\rho}-\half{1}+i})^{\times M_{\rho,i}} \times \bigtimes_{\rho \in R_1} \bigtimes_{i=0}^{a_{\rho}-1} (\rho|\cdot|^{-a_{\rho}+i})^{\times M_{\rho,i}} \rtimes \pi(\phi_0, \varepsilon_0)
\end{align}
for some $\varepsilon_0 \in \widehat{\mathcal{S}}_{\phi, \chi_{V_0}}$. Here the order of the product follows the convention 
\[\bigtimes_{i=1}^{n} \rho_i= \rho_1\times \rho_2 \times \cdots \times \rho_{n}.\]
Indeed, if \eqref{eq goal AZ} holds, then the right hand side is exactly the standard module of $\widehat{\pi}$ (the subrepresentation version), which implies that $\phi_{\widehat{\pi}}= \phi_{\psi}= \widehat{\phi_\pi}$ by the compatibility between Langlands classification and local Langlands correspondence (see \cite[Theorem 3.3(5)]{Ato20} for example). We apply induction on $\dim(\phi_1)$ to prove \eqref{eq goal AZ}. 

Suppose $\phi_1=0$. Then $\phi= \psi$ is both tempered and anti-tempered. In this case $\widehat{\pi} \in  \Pi_{\psi}=\Pi_{\phi}$, which is also equal to the $L$-packet $\Pi_{\phi_{\psi}}$. This verifies \eqref{eq goal AZ} in this case. Suppose $\phi_1 \neq 0$. Then we may construct $\pi^-= \pi(\phi^-, \varepsilon^{-})$ as in Working Hypothesis \ref{assu AZ} with
\[ \phi^-= \phi - (\rho \otimes S_{A_{\rho}} \otimes S_1)^{ \oplus m_{\rho,a_{\rho}}}+ (\rho \otimes S_{A_{\rho}-2} \otimes S_1)^{ \oplus m_{\rho,a_{\rho}}},\]
where $A_{\rho}= 2a_{\rho}$ if $\rho \in R_{0}$ and $A_{\rho}= 2a_{\rho}+1$ if $\rho \in R_{1}$. Note that $\varepsilon^{-}$ is also trivial and hence $\pi^-$ is also generic. Working Hypothesis \ref{assu AZ} implies
\[ \pi \hookrightarrow \underbrace{\rho|\cdot|^{\half{A_{\rho}-1}}\times \cdots \times\rho|\cdot|^{\half{A_{\rho}-1}}}_{m_{\rho, a_{\rho}} \text{ copies}}\rtimes \pi^-.\]
Then (see the Algorithm in the introduction of \cite{Jan18}, \cite[(1.2)]{Jan18} and \cite[Theorem 8.3.4]{MR18})
   \begin{align}\label{eq AZ final}
       \widehat{\pi} \hookrightarrow \underbrace{\rho|\cdot|^{-\half{A_{\rho}-1}}\times \cdots \times\rho|\cdot|^{\half{A_{\rho}-1}}}_{m_{\rho, a_{\rho}} \text{ copies}} \rtimes \widehat{\pi^-} \hookrightarrow \underbrace{\rho|\cdot|^{-\half{A_{\rho}-1}}\times \cdots \times\rho|\cdot|^{\half{A_{\rho}-1}}}_{m_{\rho, a_{\rho}} \text{ copies}}\rtimes M(\widehat{\pi^-}),
   \end{align}
   where $M(\widehat{\pi^{-}})$ is the standard module of $\widehat{\pi^{-}}$. Here we use the fact that $m_{\rho,a_{\rho}}=M_{\rho,a_{\rho}}$ and the assumption that $ {}^\sigma \rho^{\vee}= \rho$ since $\phi$ is of good parity. The induction hypothesis on $\pi^-$ implies that $M(\widehat{\pi^-})$ should match the right hand side of \eqref{eq goal AZ} for $\pi^-$. Thus, comparing $\phi_{\psi}$ and $\phi_{\psi^{-}}$, one can see that \eqref{eq AZ final} is exactly \eqref{eq goal AZ} for $\pi$. This completes the proof of the proposition.   
\end{proof}

\end{document}